\documentclass[a4paper,11pt]{article}
\usepackage[T1]{fontenc} 
\usepackage{times}
\usepackage{a4wide}
\usepackage{amsthm, amsmath, amssymb}
\usepackage{algorithm}
\usepackage[noend]{algpseudocode}
\usepackage[british]{babel}
\usepackage{cases}
\usepackage{dsfont}
\usepackage{array}
\usepackage{chngcntr}
\usepackage{mathrsfs}
\usepackage{hyperref}
\usepackage[authoryear]{natbib}
\usepackage{fancyvrb}
\usepackage{bbm}
\usepackage{bm}
\usepackage{latexsym}
\usepackage[pdftex]{xcolor, graphicx}
\usepackage[normalem]{ulem}
\usepackage[font={small}]{caption}
\usepackage{lipsum}                    	
\usepackage{verbatim}
\usepackage{url}
\usepackage{authblk}

\usepackage{accents}

\newlength{\dhatheight}
\newcommand{\doublehat}[1]{%
    \settoheight{\dhatheight}{\ensuremath{\hat{#1}}}%
    \addtolength{\dhatheight}{-0.35ex}%
    \hat{\vphantom{\rule{1pt}{\dhatheight}}%
    \smash{\hat{#1}}}}

\newtheorem{theorem}{Theorem}
\makeatletter
\renewcommand\@biblabel[1]{-}
\makeatother
\theoremstyle{remark}
\newtheorem{remark}{Remark}

\theoremstyle{plain}
\newtheorem{corollary}{Corollary}
\newtheorem{lemma}{Lemma}
\newtheorem{assumption}{Assumption}
\newtheorem{definition}{Definition}
\newtheorem{proposition}{Proposition}
\definecolor{jpred}{rgb}{0,0.7,0}

\def\argmax{\mathop{\rm argmax}}

\newcommand\ltwomu{L^2_{\mu}}
\newcommand\ltwom{L^2_{M}}
\newcommand\ltwoe{L^2_{e}}
\newcommand\ltwomn{L^2_{M,n}}
\newcommand\ltwoen{L^2_{e,n}}
\newcommand\ltwoeNn{L^2_{e,N-n}}
\newcommand\ltwoeNnp{L^2_{e,N-n-1}}

\newcommand\ltwoenp{L^2_{e,n+1}}
\newcommand\ltwoxin{L^2_{\xi,n}}

\newcommand\ltwochin{L^2_{\chi, n}}
\newcommand\ltwochinp{L^2_{\chi, n+1}}
\newcommand\ltwofn{L^2_{f, n}}
\newcommand\ltwofnl{L^2_{f, n-1}}
\newcommand\ltwofnp{L^2_{f, n+1}}
\newcommand\ltwoFn{L^2_{F, n}}

\newcommand\dotss{,\,\dots,\,}
\newcommand\domU{\mathcal{D}^U}
\newcommand\tcX{\tilde{\mathfrak{X}}}
\newcommand\tcY{\tilde{\mathfrak{Y}}}

\title{Regress-Later Monte Carlo for optimal control  of Markov processes}

\author{Alessandro Balata\thanks{email: A.Balata@leeds.ac.uk} }
\author{Jan Palczewski\thanks{email: J.Palczewski@leeds.ac.uk}}
\affil{School of Mathematics, University of Leeds, LS2 9JT, Leeds, United Kingdom}

\begin{document}

\maketitle

\begin{abstract}
We develop two Regression Monte Carlo algorithms (value and performance iteration) to solve general problems of optimal stochastic control of discrete-time Markov processes. We formulate our method within an innovative framework that allow us to prove the speed of convergence of our numerical schemes. We rely on the Regress Later approach unlike other attempts which employ the Regress Now technique. We exploit error bounds obtained in our proofs, along with numerical experiments, to investigate differences between the value and performance iteration approaches. Introduced in \citet{Tsitsiklis2001} and \citet{Longstaff2001} respectively, their characteristics have gone largely unnoticed in the literature; we show however that their differences are paramount in practical solution of stochastic control problems. Finally, we provide some guidelines for the tuning of our algorithms.
\end{abstract}


\section{Introduction}
\label{introduction}
In this paper we will introduce and prove the convergence of two Regress-Later Monte Carlo schemes for the solution of discrete-time general Markovian stochastic control problems. 

Let us consider a controlled Markov process $(X_n)_{n=0}^N$ on a domain $\mathcal{D} \subseteq$ $\mathbb{R}^d$  specified as follows:
\begin{equation}
\label{eq:Xdynamics}
X_{n+1} = \varphi (n, X_n, \xi_n, u_n),
\end{equation}
where $\varphi$ is a Borel-measurable function and $\{\xi_n\}$ is a collection of i.i.d. uniformly distributed random variables on $[0,1]$. Without loss of generality, we assume that controls $(u_n)$ are in a feedback form, i.e.,
\[
u_n = u_n (X_n)
\]
and belong to a compact set $\domU \subset \mathbb{R}^q$. We will denote the set of controls of the above form by $\mathcal{U}$. Notice that due to $\varphi$ depending on $n$ and $X_n$, our setting can accomodate sets of controls that are state dependent. 

In this setting, we define a pathwise performance measure
\begin{equation}
\label{eq:Jperformance}
J(n,(X_s,u_s)_{s=n,\,\dots,\,N})=\sum_{s=n}^{N-1} f(s,X_s,u_s)+g(X_N),\quad n=0\dots N,
\end{equation}
where $g$ is the terminal condition and $f$ is the running reward. 
We want to study the problem of computing the value function and an optimal control coresponding to the following optimisation problem:
\begin{equation}
\label{eq:problem}
V(n,x)=\sup_{u\in \mathcal{U}}\mathbb{E}\Bigl[ J(n,(X_s,u_s)_{s=n,\,\dots,\,N}) |X_n=x\Bigr].
\end{equation}

A convenient equivalent representation of   \eqref{eq:problem} is given in terms of dynamic programming equation which allows to recursively compute the value function backward in time from the known terminal condition:
\begin{equation}
\label{eq:dynamicvalue}
\begin{cases}
V(N,x)=g(x),&\\
V(n,x)=\max\limits_{u\in \domU}\Bigl\{ f(n,x,u)+\mathbb{E}\bigl[V(n+1,X_{n+1}) |X_n=x,u_n=u\bigr]\Bigr\}.&
\end{cases}
\end{equation}
The dynamic programming equation inspires a numerical method to compute the value function and the control: starting from the known terminal condition at time $N$ we recursively compute the value function, backward in time, optimising the one step performance measure. The main difficulty in implementing such a strategy is the estimation of the conditional expectation $\mathbb{E}[V(n+1,X_{n+1}) |X_n=x,u_n=u]$.

The Regression Monte Carlo scheme has been pupularised by \citet{Longstaff2001} and \citet{Tsitsiklis2001}, however it is the result of the contribution of different papers among which \citet{Carriere1996} and \citet{Broadie2000}. The successful idea behind Regression Monte Carlo is to approximate the conditional expectations appearing in the dynamic programming equation \eqref{eq:dynamicvalue} with a projection on the space generated by a set of basis functions. In practice, a set of Monte Carlo simulated trajectories is used to iterate the dynamic programming equation and to estimate the regression coefficients by approximating expectations with sample averages. Regression Monte Carlo has been successfully applied to the evaluation of complex financial derivatives, including American options. However, the standard approach cannot be applied to \ref{eq:problem} where the control directly affects the dynamics of the state processes.

Contrary to the case of option pricing, or uncontrolled dynamics in general, in which the conditional expectation of $V(n+1,X_{n+1}) |\,X_n=x$ can be estimated from the cross sectional information contained in simulated trajectories,  in the case of controlled Markov processes such trajectories depend on the control process $(u_n)$ and cannot therefore be simulated beforehand. In particular, given that the forward trajectories should be computed for a fixed values of the control, the estimated conditional expectation will be relevant only for that particular choice of the control. 

A different approach it that of approximating the discrete time control problem with a continuous time one, as long as $T/N$ is small. Under mild conditions a continuous time control problem can be solved using Pontryagin principle and be reformulated in terms of a system of coupled FBSDEs. The advantage of this formulation is that the effect of the control is partially decoupled from the controlled process, allowing, after a further time discretization, to solve the problem using traditional regression Monte Carlo.  
The literature in this area is vast, note the French school in particular with \citet{Gobet2005} and \citet{Gobet2016} and the review given in \citet{Bender2012}, for example.

Our approach however is different. We maintain the original formulation of the discrete-time problem and solve directly the dynamic programming equation by exploiting the characteristic of a particular variant of regression Monte Carlo called Regress Later. We introduced this approach in the particular framework of control of degenerate processes in \cite{Balata2017i}.
Regress-later approximations can be traced back to \citet{Broadie2000}, \citet{Glasserman2002} and \citet{Broadie2004}, and more recently were studied in \citet{Beutner, Jain2015}. In those papers the regress-later approach is regarded as a tool to reduce the approximation error in traditional exogenous Regression Monte Carlo problems (e.g., American option pricing). To our knowledge, its ground-breaking potential for problems with endogeneous (controlled) state variables has not been recognised yet. A recent application to the solution of systems of FBSDEs can be found in \citet{Briand2014} and \citet{Pelsser2017}.
In applications, \citet{Secomandi2014} compare regress-now and regress-later estimates in the context of energy real options, while \citet{Selvaprabu2017} explore the links between Regression Monte Carlo and Approximate Dynamic Programming.

One main alternative has been proposed in order to generalise the standard Regression Monte Carlo method to problems of control of Markov processes: the control randomisation approach proposed in \cite{Kharroubi2013} and \citet{Langrene2015}. The technique makes up for the limitations of the traditional Regress Now approach by explicitly introducing dependence on the control in the basis functions, in turn obtaining an estimated conditional expectation that depends on the choice of the control. In order for the regression approximation to have the correct statistical properties, an initial set of random trajectories of the control should be simulated and then used in the estimation of the projection coefficients. A comparison between the two methods can be found in \cite{Balata2018stochmkv}. General conclusions of this study are that Regress Later is faster (fewer arguments of basis functions) and easier to tune than Control Randomisation which is highly dependent on the choice of the initial randomised control.

Proofs of convergence and error bounds for different specifications of Regression Monte Carlo have appeared in the literature over years, c.f. \citet{Clement2002} and \citet{Beutner}. All the available studies, however, deal with the uncontrolled dynamics. In the case of control randomisation, a proof of convergence is not available. Error bounds for continuous-time stochastic control problems are, however, available when the Pontryagin principle is employed to rewrite the problem as a system of FBSDEs, see \citet{Lemor2006}.

A second, largely unnoticed, characteristics of regression Monte Carlo algorithms is the actual function to be projected backward in time by the regression approximation. \citet{Tsitsiklis2001} proposed the well known value iteration approach, which directly follows from the dynamic programming equation \eqref{eq:dynamicvalue} and consist of projecting the estimated value function backward in time. \citet{Longstaff2001}, on the other hand, proposed a so-called ``policy iteration/path recomputation/performance iteration'' approach which computes first the pathwise values of a sequence of decisions from different starting point, and then project them backward in time. The two approaches have been studied in the framework of optimal stopping problems by \citet{Egloff2007}, which introduces an hybrid method that takes advantage of the low bias of performance iteration and low variance of value iteration. To the best of our knowledge however a systematic analysis of the differences between the two techniques, in the context of controlled Markov processes in particular, has never been carried over. 

The contribution of this paper is threefold: we give a systematic description of a powerful but relatively simple algorithm to solve general problems of stochastic control of discrete-time Markov processes providing theoretical and empirical results. The Regression Monte Carlo approach relies on the choice of basis functions and a training measure, and we provide guidance on the selection of choices. We prove convergence and derive error bounds of two Regress Later-based numerical schemes, enriching the literature with both new, effective and provably convergent numerical schemes, and a new framework within which convergence of different Regression Monte Carlo schemes can be proved. We give theoretical and heuristic evidence of the difference between value and performance iteration schemes through our error bounds, and through a numerical example that showcases the most interesting characteristics of the two types of iteration.

The rest of the paper is organised as follows: in Section \ref{s:assumptions} we present the mathematical framework and the assumptions needed for proving the convergence of the numerical scheme. Section \ref{s:RLMC} follows with a rigorous presentation of Regress Later in both the value and performance iteration specification and convergence theorems. The last part of this section includes a discussion of differences between value and performance iteration, along with some numerical examples. In Section \ref{proofs} we collect proofs of the main theorems. Finally, in Section \ref{sec:numerics} we present two numerical experiments, the first aimed at showing that the algorithms converge to the exact solution as expected, the second aimed at highlighting the differences between value and performance iteration schemes. Conclusions are drown in Section \ref{s:conclusion}

\section{Assumptions and preliminary results}
\label{s:assumptions}
In this section we present some of the standing Assumptions of the paper and some methodological results we will use later in the paper.

\subsection{Assumptions}
Let $\mu$ be a probability measure on the space $\mathcal{D}$. We will sometimes refer to it as the \emph{training distribution}.
\begin{assumption}
\label{a:density}
We assume that the process $X$ has a transition density with respect to the measure $\mu$, i.e.,
\[
\mathbb{P}_{\mu}\big(X_{n+1}\in A|X_{n}=x, u_n = u \big)=\int_{A} r(n,x,u;y)\mu(dy)
\]
and, in addition, this density is uniformly bounded
\[
r(n, x, u; y) \le \bar{R}^2, \qquad \forall\ n,x,u,y.
\]
\end{assumption}

\begin{remark}
Assumption \ref{a:density}, in most cases, is satisfied only when compact domains are considered. Therefore, even though we do not need to explicitly assume compactness of the domain $\mathcal{D}$ in our proofs, truncation could be necessary.
\end{remark}

\begin{assumption}
\label{a:normV}
The running profit function $f$ and the terminal condition $g$ are measurable and bounded, i.e. $\|f\|_{\infty}+\|g\|_{\infty}<\infty$.
\end{assumption}

\begin{remark}
\label{r:normVa}
The value function $V(n,x)$ is bounded, i.e. $V(n,x)<\Gamma$ for all $n=1,\,\dots,\,N$ and $x\in \mathcal{D}$. A trivial bound follows from the boundedness of $f$ and $g$:  $\Gamma\le\bar{\Gamma}:=(N-1)\|f\|_{\infty}+\|g\|_{\infty}$.
\end{remark}

We will denote by $\ltwomu=L^2(\mathcal{D},\mu)$ the Hilbert space of  real-valued functions on $\mathcal{D}$ that are square integrable with respect to $\mu$.

\begin{definition}
A family of $K$ linearly independent functions $\{\phi_k(\cdot)\}_{k=1}^{K}:\mathcal{D}\to\mathbb{R}$ generating a linear subspace of $\ltwomu$ is called a family of \emph{basis functions}.
\end{definition}
Due to practical reasons that will become clear later on, we will neither assume that the functions are ortogonal, nor that their norms are equal to $1$. 

\begin{definition}\label{def:hat_phi}
Denote by $\hat{\phi}^n_k$, $k=1\dotss K$, $n=0 \dotss N-1$:
\[
\hat{\phi}_k^n\left(x,u\right) = \mathbb{E}\left[\phi_k\left(X_{n+1}\right)|X_n=x,u_n=u\right]= \int_\mathcal{D} \phi_k(y) r(n, x, u; y) \mu(dy).
\] 
\end{definition}

\begin{assumption}
\label{a:semicont_hat_phi}
The functions $f$ and $\hat \phi_k^n$ are upper semi-continuous in $u$.
\end{assumption}

\begin{remark}
The upper semi-continuity requested in Assumption \ref{a:semicont_hat_phi} is required to assert existence of an optimiser in our algorithms. We impose it only for convenience of notation and proofs, but it can be relaxed easily by considering $\varepsilon$-optimisers for a sufficiently small $\varepsilon > 0$ and obvious modifications of statements of error bounds.
\end{remark}

%

\subsection{Random projection operator}

Let us introduce now the exact projection operator $\Pi_{K}$ on $\ltwomu$ which acts projecting its argument onto the space generated by the basis functions, i.e. $lin(\phi_1\dotss\phi_K)\subset\ltwomu$. For $h\in \ltwomu$, we have $\Pi_{K}h=\sum_{k=1}^K\alpha_k \phi_k$ with the coefficients $\pmb{\alpha} = (\alpha_1, \ldots, \alpha_K)^T$ given by 
\begin{equation}
\label{eq:alpha}
\pmb{\alpha}=\mathcal{A}_K^{-1}\big\langle h,\pmb{\phi}\big\rangle_{\ltwomu},
\end{equation}
where $\pmb{\phi}=(\phi_1\dotss\phi_K)^T$ and $\mathcal{A}_K=\big\langle\pmb{\phi},\pmb{\phi}^T \big\rangle_{\ltwomu}$. The scalar product in $\ltwomu$ can be written as an expectation with respect to $\mu$, in the sense that:
\begin{equation}
\label{eq:scalar}
\big\langle h,\pmb{\phi}\big\rangle_{\ltwomu}=\mathbb{E}_{\tilde{X}\sim\mu}\Big[h(\tilde{X})\pmb{\phi}(\tilde{X})\Big], \text{ and }
 \mathcal{A}_K=\mathbb{E}_{\tilde{X}\sim\mu}\Big[\pmb{\phi}(\tilde{X})\pmb{\phi}(\tilde{X})^T\Big].
\end{equation}
This guides us to a Monte Carlo estimator of $\pmb{\alpha}$. We draw $M$ i.i.d. copies $\tilde{X}^1, \ldots, \tilde{X}^M$ of $\tilde{X} \sim \mu$ which we call the \emph{training points}. For $h\in\ltwomu$ we approximate $\pmb{\alpha}$ by
\begin{equation}
\label{eq:est_alpha}
\hat{\pmb{\alpha}}=\mathcal{A}_K^{-1}\frac{1}{M}\sum_{m=1}^M\big[h(\tilde{X}^m)\pmb{\phi}(\tilde{X}^m)\big] 
\end{equation}
and define the \emph{random projection operator} $\hat{\Pi}_{K} h=\sum_{k=1}^K\hat{\alpha}_k \phi_k$. 
\begin{remark}
In formula \eqref{eq:est_alpha}, we assume that $\mathcal{A}_K$ can be evaluated exactly (or precomputed with a very high precision) as it depends only on our choice of basis functions and the measure $\mu$. This compares favourably (in terms of speed and accuracy) to classical regression Monte Carlo in which both expectations in \eqref{eq:scalar} have to be approximated at each time step via Monte Carlo averages.
\end{remark}

Denote by $\ltwom$ the space linked to the training points, $\ltwom = L^2 (\mathcal{D}^M, \mu^M)$, and we write $\ltwoe=\ltwom\times\ltwomu$. 
Notice that we have $\hat{\Pi}_{K}h\in \ltwom \times lin(\phi_1, \ldots, \phi_K) \subset\ltwoe$ because the random projection coefficients are functions of $(\tilde{X}^1\dotss \tilde{X}^M)$, i.e. $\hat{\pmb{\alpha}}=\hat{\pmb{\alpha}}(\tilde{X}^1\dotss \tilde{X}^M)\in(\ltwom)^K$.

\subsection{Extension of  random projection operator}

We extend the projection operator introduced above to functions living in spaces bigger than $\ltwomu$. This is to introduce the notation which we will need later and does not involve any further mathematical complications. Define the space $\ltwomn = L^2(D^{Mn}, \mu^{Mn})$ generated by $n$ collections of M training points, $n=1\dotss N$, denoted by $\{\tilde{X}^m_s\}_{m=1,s=N-n+1}^{M,N}$. The unusual indexing is related to times at which training points are placed while iterating backwards through the dynamic programming equation \eqref{eq:dynamicvalue}. We will write $\ltwoen=\ltwomn\times\ltwomu$ with $\ltwoen \ni h=h(\tcX_{n};\tilde{X})$, where we identify $\tcX_{n} = \{\tilde{X}^m_s\}_{m=1,s=N-n+1}^{M,N}$ with the first set of coordinates corresponding to $\ltwomn$ and $\tilde X$ with the remaining coordinate of $\ltwomu$. For the brevity of notation, we will write , $L^2_{e,1}=L^2_{e}$, $L^2_{e,0}=L^2_{\mu}$, and $\tcX_0 = \{\emptyset \}$. When evaluating the norm $\| \cdot \|_{\ltwoen}$ we will often denote by `$(\cdot)$' the argument corresponding to the $\ltwomu$ component.

Define an extended projection operator as
\[
\Pi_K^{N-n} h(\tcX_n\,;\cdot)=\big(\pmb{\alpha}^{N-n}\big)^T\, \pmb{\phi} (\cdot),
\]
where
\[
\big(L^2_{M,n}\big)^K \ni \pmb{\alpha}^{N-n}=\pmb{\alpha}^{N-n} \big(\tcX_n\big)
= \mathcal{A}_K^{-1}\mathbb{E}_{\tilde{X}\sim\mu}\Big[ h(\tcX_n; \tilde X)\pmb{\phi}(\tilde{X})\Big| \tcX_n\Big].
\]
Notice that  $\Pi^{N-n}_K h \in \ltwoen$ since the coefficients $\pmb\alpha^{N-n}$ still depend on the randomness contained in $\tcX_n = \{\tilde{X}^m_s\}_{m=1,s=N-n+1}^{M,N}$. The superscript $N-n$ in $\Pi^{N-n}_K h$ indicates the dependence on $\tcX_n$. However, from a mathematical perspective, for fixed $\tcX_n$ the operator $\Pi_K^{N-n}$ is identical to $\Pi_K$, and, indeed, it can be defined pointwise for each $\tcX_n$.

Similarly as above we define the ``random projection operator'' acting on $h\in\ltwoen$ by
\[
\hat\Pi^{N-n}_K h(\tcX_n\,;\cdot)=\big(\hat \alpha^{N-n}\big)^T\, \pmb{\phi} (\cdot),
\]
where
\[
\hat{\pmb{\alpha}}^{N-n}=\hat{\pmb{\alpha}}^{N-n}(\tcX_n,\{\tilde{X}^m_{N-n}\}_{m=1}^{M})=\mathcal{A}_K^{-1}\frac{1}{M}\sum_{m=1}^M h(\tcX_n;\tilde{X}^m_{N-n})\pmb{\phi}(\tilde{X}^m_{N-n}),
\]
and $\{\tilde X^m_{N-n}\}_{m=1}^M$ are i.i.d. random variables with the distribution $\mu$.
It follows that $\hat{\Pi}^{N-n}_K h \in \ltwoenp$. Notice then that the random projection operator produces functions which live in a bigger space than the space where $h$ lives, in particular every projection adds one layer of training points so that the original space is enlarged by the addition of $\ltwom$. For a graphical representation of the spaces introduced in this section see Figure \ref{f:spaces}.

\begin{figure}
\hspace{110pt}\def\svgwidth{0.6\linewidth} 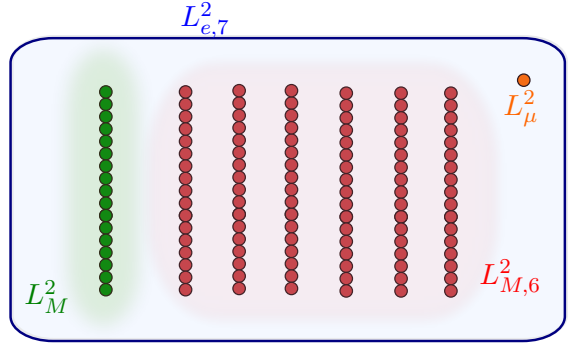
\caption{In the figure above we give, for the convenience of the reader, a graphical representation of the square integrable spaces introduced so far. Notice that every element is given by the measure $\mu$, a different number of them is considered in the different spaces.} 
\label{f:spaces}
\end{figure}

\subsection{Preliminary results}

We present now two useful results relating the exact and random projection operators introduced above.

\begin{lemma}[Projection error]
\label{l:proj_error}
For $h\in \ltwoen$, the error of the random projection operator is bounded as follows:
\[
\Big\|\hat{\Pi}^{N-n}_K h (\tcX_n; \cdot) -\Pi^{N-n}_K h (\tcX_n; \cdot) \Big\|_{\ltwoe}\le \big\|\mathcal{A}_K^{-1/2}\big\|_2\frac{1}{\sqrt{M}}SDev_{\tilde{X}\sim\mu}\Big(h(\tcX_n\,;\tilde{X})\pmb{\phi}(\tilde{X})\Big),
\]
where $\tcX_n$ is assumed fixed, 
\[
SDev_{\tilde{X}\sim\mu}\Big(h(\tcX_n\,;\tilde{X})\pmb{\phi}(\tilde{X})\Big) :=
\left( \sum_{k=1}^K Var_{\tilde{X}\sim\mu}\Big(h(\tcX_n\,;\tilde{X})\phi_k(\tilde{X})\Big) \right)^{1/2},
\]
and $\| \mathcal{A}_K^{-1/2} \|_2 = \max \{ \|\mathcal{A}_K^{-1/2} x\|_2:\ x \in \mathbb{R}^K, \ \|x\|_2 = 1 \}$ is the matrix operator norm of $\mathcal{A}_K^{-1/2}$ and $\|x\|_2$ is the Euclidean norm of $x$.
\end{lemma}
\begin{proof}
By the definition of projection operators we have
\begin{align*}
&\Big\|\hat{\Pi}^{N-n}_K h (\tcX_n; \cdot)-\Pi^{N-n}_K h (\tcX_n; \cdot)\Big\|_{\ltwoe}\\
&=\Big\|\Big(\mathcal{A}_K^{-1}(\frac{1}{M}\sum_{m=1}^M h(\tcX_n;\tilde{X}^m_{N-n}) \pmb{\phi}(\tilde{X}^m_{N-n})-\mathbb{E}_{z\sim\mu}\big[h (\tcX_n; z) \pmb{\phi} ( z)\big])\Big)^T\pmb{\phi} ( \cdot)  \Big\|_{\ltwoe}\\
&=\Big\|\Big(\frac{1}{M}\sum_{m=1}^M h (\tcX_n; \tilde{X}^m_{N-n}) \pmb{\phi}(\tilde{X}^m_{N-n})-\mathbb{E}_{z\sim\mu}\big[h(\tcX_n,z) \pmb{\phi}(z)\big]\Big)^T\mathcal{A}_K^{-1}\pmb{\phi} (\cdot)  \Big\|_{\ltwoe}
=\Big\| \pmb{\beta}_n^T\mathcal{A}_K^{-1} \pmb{\phi} (\cdot)\Big\|_{\ltwoe},
\end{align*}
where $ \pmb{\beta}_n=\frac{1}{M}\sum_{m=1}^M h (\tcX_n;\tilde{X}^m_{N-n}) \pmb{\phi}(\tilde{X}^m_{N-n})-\mathbb{E}_{z\sim\mu}\big[h (\tcX_n; z) \pmb{\phi} (z)\big]$ and we used that $\mathcal{A}_K$ is symmetric.
We have
\begin{align*}
\Big\| \pmb{\beta}_n^T\mathcal{A}_K^{-1} \pmb{\phi} (\cdot) \Big\|^2_{\ltwoe}
&=\mathbb{E}_{e}\Big[ \pmb{\beta}_n^T\mathcal{A}_K^{-1} \pmb{\phi} (\cdot) \pmb{\phi} (\cdot)^T\mathcal{A}_K^{-1} \pmb{\beta}_n\Big]
=\mathbb{E}_{M}\Big[ \pmb{\beta}_n^T\mathcal{A}_K^{-1}\mathbb{E}_{\mu}[ \pmb{\phi} (\cdot)\pmb{\phi} (\cdot)^T]\mathcal{A}_K^{-1}\pmb{\beta}_n\Big]\\
&=\mathbb{E}_{M}\Big[ \pmb{\beta}_n^T\mathcal{A}_K^{-1}\pmb{\beta}_n\Big]
= \mathbb{E}_{M}\Big[ \big\| \mathcal{A}_K^{-1/2}\pmb{\beta}_n\big\|_2^2\Big]
\le \big\| \mathcal{A}_K^{-1/2}\big\|_2^2\ \mathbb{E}_{M}\Big[ \|\pmb{\beta}_n\|_2^2\Big]\\
&
= \big\|\mathcal{A}_K^{-1}\big\|_2^2\; \frac{1}{M} Var_{\tilde X \sim \mu} \Big(h(\tcX_n\,;\tilde{X})\pmb{\phi}(\tilde{X})\Big),
\end{align*}
where in the last equality we used that $\{\tilde X^m_{N-n}\}_{m=1}^M$ are independent and distributed as $\mu$, and that $\mathbb{E}_M [\pmb{\beta}_n] = 0$.
\end{proof}

\begin{lemma}[Standard Deviation]
\label{l:STDEV}
For a bounded function $h \in \ltwomu$ we have:
\[
SDev_{\tilde{X}\sim\mu}\Big(h(\tilde{X}) \pmb{\phi}(\tilde{X}) \Big)\le \sqrt{K}\|h\|_{\infty} \max\limits_{k=1,\,\dots,\,K}\|\phi_k\big\|_{\ltwomu}.
\]
\end{lemma}
\begin{proof}From the definition in Lemma \ref{l:proj_error}
\begin{equation}\label{eqn:aa1}
SDev_{\tilde{X}\sim\mu}\Big( h(\tilde{X}) \pmb{\phi}(\tilde{X}) \Big)
=\Big(\sum_{k=1}^K  Var_{\tilde{X}\sim\mu}\Big[h(\tilde{X}) \phi_k(\tilde{X}) \Big]\Big)^{\frac{1}{2}}.
\end{equation}
We bound now the variance by the second moment, and using that $\|h\|_{\infty} < \infty$ we have for each $k$
\[
\begin{split}
 Var_{\tilde{X}\sim\mu}\Big[h(\tilde{X}) \phi_k(\tilde{X}) \Big]
&\le \mathbb{E}_{\mu}\Big[\big(h(\tilde{X}) \phi_k(\tilde{X})\big)^2 \Big] \le \|h\|_{\infty}^2\|\phi_k \|_{\ltwomu}^2.
\end{split}
\]
Inserting this bound into \eqref{eqn:aa1} completes the proof.
\end{proof}

\begin{remark}[Norm of the true projection operator]
\label{l:proj_bound}
Let $h$ be an element of $\ltwoen$. The true projection operator $\Pi^{N-n}_{K}$ admits the following bound:
\[
\Big\|\Pi^{N-n}_K h (\tcX_n; \tilde{X}) \Big\|_{\ltwoen}\le\Big\|  h (\tcX_n; \tilde{X}) \Big\|_{\ltwoen}.
\]
\end{remark}

\begin{lemma}[Bound on conditional expectation]
\label{l:traj}
For any $h\in \ltwoen$, we have the following bound on the norm of the conditional expectation
\[
\begin{split}
\Big\| \sup_{u \in \mathcal{D}^U} \mathbb{E}\big[h (\tcX_n; X_{n+1})\big|X_n=\tilde{X}, u_n = u) \big]\Big\|_{\ltwoen}
&\le\bar{R} \| h (\tcX_n; \tilde{X})\|_{\ltwoen},
\end{split}
\]
where 
\[
\mathbb{E}\big[h (\tcX_n; X_{n+1})\big|X_n=x, u_n = u \big] = \int_{\mathcal{D}} h (\tcX_n; y)\; r(n, x, u; y) \mu(dy). 
\]
\end{lemma}
\begin{proof}
Using Jensen inequality and Assumption \ref{a:density}:
\[
\begin{split}
\Big\| \sup_{u \in \mathcal{D}^U} \mathbb{E}\big[h (\tcX_n; X_{n+1})\big|X_n=\tilde{X}, u_n = u) \big]\Big\|_{\ltwoen}^2
&= \mathbb{E}_{M,n}\Big(\sup_{u \in \mathcal{D}^U} \int_\mathcal{D} r(n,\tilde{X},u;y)h (\tcX_n; y)\mu(dy)\Big)^2\\
&\le\mathbb{E}_{M,n}\Big(\sup_{u \in \mathcal{D}^U} \int_\mathcal{D} r(n,\tilde{X},u;y)h^2 (\tcX_n; y)\mu(dy)\Big)\\
&\le\bar{R}^2\; \Big\| h (\tcX_n; \tilde{X})\Big\|_{\ltwoen}^2,
\end{split}
\]
where $\mathbb{E}_{M,n}$ is the expectation with respect to the measure underlying the space $\ltwomn$.
\end{proof}

\section{Regress Later Monte Carlo}
\label{s:RLMC}
This section presents the main contribution of the paper: value and performance iteration algorithms as well as estimates of their convergence speed.


In the classical framework the regression was used to approximate the conditional expectation with respect to $X_n$ directly as a linear combination of basis functions of this variable (e.g., in \eqref{eq:dynamicvalue}). As opposed to this \emph{regress-now} approach, we employ a \emph{regress-later} idea in which conditional expectation with respect to $X_n$ is computed in two stages. First, a conditional expectation with respect to $X_{n+1}$ is approximated in a regression step by a linear combination of basis functions of $X_{n+1}$. Then, analytical formulas are applied to condition this linear combination of functions of future values on present values $X_n$ and a control $u_n$.


\subsection{Value iteration}
\label{s:value}

\begin{algorithm}[tb]
 \caption{Regress-later Monte Carlo algorithm (RLMC) - Value iteration}
 \label{alg:RLMCV}
 \textbf{input:} M, K, $\mu$, $\{\phi\}_{k=1}^K$
 \begin{algorithmic}[1]
\State Pre-compute the inverse of the covariance matrix $\mathcal{A}_K$
 \State Generate i.i.d. training points $\{\tilde{X}_N^{m}\}_{m=1}^{M}$ with the distribution $\mu$.
\State Initialise the value function $\hat V (N, \tilde{X}_{N}^{m}) =g(\tilde{X}_{N}^{m}), \quad  m = 1, \ldots, M.$
 \For{n=N-1 to 0}
\State  $\hat{\pmb{\alpha}}^{n+1}=\mathcal{A}_K^{-1}\frac{1}{M}\sum_{m=1}^{M}\Big[\hat V(n+1,\tilde{X}_{n+1}^{m}) \pmb{\phi}(\tilde{X}_{n+1}^{m})\Big]$
 \State Generate a new layer of i.i.d. training points $\{\tilde{X}_n^{m}\}_{m=1}^{M}$ with the distribution $\mu$.
 \State For all $m$ do
\[
\hat V(n,\tilde{X}_{n}^{m})=\Bigg(\sup_{u \in \domU} \Big\{f(n,\tilde{X}_{n}^{m},u)+\sum_{k=1}^K\hat\alpha^{n+1}_k\hat{\phi}^{n}_k(\tilde{X}_{n}^{m}, u)\Big\}\Bigg)\wedge \Gamma \vee (-\Gamma)
\]
\EndFor
\end{algorithmic}
\textbf{output:} $\{\hat{\alpha}_{n}^k\}_{n,k=1}^{N,K}$
 \end{algorithm}

We will now present regress-later solution to \emph{value iteration} procedure. This approach follows from the most natural approximation of the dynamic programming equation \eqref{eq:dynamicvalue}, in which the conditional expectation of the value function at the next time step is replaced by  its estimator.

We start from time $N$, when the terminal condition is known and we set $\hat{V}(N,x)=V(N,x) = g(x)$, $\forall \,x\in\mathcal{D}$. We move now to time $N-1$. The dynamic programming equation \eqref{eq:dynamicvalue} requires us to compute $\mathbb{E}[V(N,X_N)|X_{N-1}=x,u_{N-1}=u]$. In order to do so we generate $M$ samples $\{\tilde{X}^m_{N}\}_{m=1}^M$ from the distribution $\mu$. These are used for the estimation of the projection coefficients $\hat{\pmb{\alpha}}^{N}$ in the random projection operator $\hat{\Pi}_K\hat{V} = \hat{\Pi}^N_K\hat{V}\in\ltwoe$, i.e., we compute conditional expectation by first projecting $\hat{V}(N, \tilde X)\in\ltwomu$ over the basis functions $\{\phi_k\}_{k=1}^{K}$. Then we compute analytically conditional expectations of the obtained linear combination of basis functions:
\begin{align*}
\mathbb{E}_{N-1, x,u}[V(N,X_N)]\approx\mathbb{E}_{N-1, x,u}[\hat{\Pi}^N_K\hat{V}(N,X_N)] &=\sum_{k=1}^K\hat \alpha_k^{N}\mathbb{E}_{N-1, x,u}[\phi_k(X_N)]\\
&= \sum_{k=1}^K\hat \alpha_k^{N} \hat{\phi}_k^{N-1}\left(x,u\right),
\end{align*}
where $\hat{\phi}_k^{N-1}\left(x,u\right) = \mathbb{E}_{N-1, x, u}[\phi_k(X_N)]$.
We then set
\[
\hat V(N-1,x)=\sup_{u \in \domU} \Big\{f(N-1,x,u)+\sum_{k=1}^K\hat\alpha^{N}_k\hat{\phi}^{N-1}_k(x, u)\Big\},
\]
and we move to next time step $N-2$ with the function $\hat{V}(N-1,x)\in\ltwoe = L^2_{e,1}$ (due to the randomness introduced by the training points $\{\tilde X^m_N\}_{m=1}^M$). Similarly to the previous time step, we project the value function using the random projection operator obtaining $\hat{\Pi}^{N-1}_K\hat{V}(N-1,x)\in L^2_{e,2}$ from which we can compute an estimator of the conditional expectation. 

In general, the approximate value function $\hat{V}(n,\tilde{X})\in L^2_{e,N-n}$ obeys the following dynamic programming equation (note that we will often refer to it as the backward procedure):
\begin{equation}
\label{eq:v_iter}
\begin{cases}
&\hat{V}(N,x)=g(x),\\
&\hat{V}(n,x)=\sup\limits_{u\in\domU}\Big\{f(n,x,u)+\mathbb{E}\big[\hat{\Pi}^{n+1}_K\hat{V}(n+1,\cdot)|X_n=x,u_n=u\big] \Big\} \wedge \Gamma \vee (-\Gamma)\\
&\phantom{\hat{V}(n,x)}= \sup\limits_{u\in\domU}\Big\{f(n,x,u)+\sum_{k=1}^K \hat\alpha^{n+1}_k \hat \phi^{n+1}_k (x,u) \Big\} \wedge \Gamma \vee (-\Gamma),
\end{cases}
\end{equation}
where functions $\hat \phi_k^n$ are introduced in Definition \ref{def:hat_phi}. Details of implementation are collected in Algorithm \ref{alg:RLMCV}.

\begin{remark}
\label{r:normV}
We exploit the bound in Assumption \ref{a:normV} which allows us to truncate the Monte Carlo estimate
\[
\sup\limits_{u\in\domU}\Big\{f(n,x,u)+\sum_{k=1}^K \hat\alpha^{n+1}_k \hat \phi^{n+1}_k (x,u) \Big\}
\]
in \eqref{eq:v_iter} to the inverval $[-\Gamma, \Gamma]$. The true value function satisfies these bounds, so the exceedance of this interval in the above expression may only be caused by approximation errors due to the projection on basis functions and Monte Carlo estimate of $\hat{\pmb{\alpha}}^{n+1}$.
\end{remark}

\begin{remark}
Note that as the matrix $\mathcal{A}_K$ in line 5 of the algorithm is computed with respect to the measure $\mu$, we do not need to estimate it and invert it at every time step, which  is required in traditional regression Monte Carlo methods. Rather, we can precompute it before starting the backward procedure saving computational time and improving the quality of estimations.
\end{remark}
\begin{remark}
Notice that the random coefficients $\hat{\pmb\alpha}^{n+1}$ are independent from $\{\tilde{X}^m_n\}_{m=1}^M$ and also from the law of $(X_{n+1}|X_n=\tilde{X}^m_n, u_n = u)$. Therefore we can compute the conditional expectation in \eqref{eq:v_iter}, exploiting linearity, as 
\[
\mathbb{E}\big[\hat{\Pi}^{n+1}_K\hat{V}(n+1,\cdot)|X_n=x,u_n=u\big] = \sum_{k=1}^K \hat{\alpha}^{n+1}_k\hat{\phi}^n_k (x,u).
\]
This decomposition enables our approach for optimal control of Markov processes.
\end{remark}
\begin{remark}
Both the value function $\hat V(n,x)$ and the regression coefficients $\hat{\pmb\alpha}^{n+1}$ depend implicitely on all the training points used at times $n+1, \ldots, N$, i.e., on $\tcX_{N-n}$. This dependence will be omitted in notation and only indicated in the proof by applying appropriate projection operators $\Pi_K^{n+1}$ and $\hat \Pi_K^{n+1}$.
\end{remark}

The following theorem offers an upper bound for the error between the estimated and the true value function; the proof is deferred until Section \ref{proofs}.

\begin{theorem}
Under Assumptions \ref{a:density}-\ref{a:semicont_hat_phi}  for all $n = 0, 1, \ldots, N$
\label{th:backward}
\[
\Big\|\hat{V}(n,\tilde{X})-V(n,\tilde{X}) \Big\|_{\ltwoeNn}\le  \bar{R}\frac{\bar{R}^{N-n}-1}{\bar{R}-1}\Big(\epsilon_K+\frac{\sqrt{K}}{\sqrt{M}}\Gamma\|\mathcal{A}_K^{-1/2}\|_{2}\max\limits_{n=1,\,\dots,\,N}\|\phi_k\|_{\ltwomu}\Big) ,
\]
where $\epsilon_K=\max_{n=1, \ldots, N} \Big\|\Pi^n_{K}V(n,\tilde{X})-V(n,\tilde{X}) \Big\|_{\ltwomu}$.
\end{theorem}

\begin{remark}
\label{r:back_conv}
The explicit dependence of the error bound on the number of basis functions $K$ and the accuracy $\epsilon_K$ with which they can approximate the true value function allows for derivation of the tradeoff between the number of basis functions and the number of Monte Carlo iterations.
\end{remark}

\begin{corollary}
\label{c:backward}
If, in addition, the family $\{\phi_k\}_{k=1}^K$ is orthonormal, then $\mathcal{A}_K=Id$ and  the following bound holds:
\[
\Big\|\hat{V}(n,\tilde{X})-V(n,\tilde{X}) \Big\|_{\ltwoeNn}\le \bar{R}\frac{\bar{R}^{N-n}-1}{\bar{R}-1}\Big(\epsilon_K+\frac{\sqrt{K}}{\sqrt{M}}\Gamma\Big)
\]
\end{corollary}
\begin{proof}
Since the basis functions are normalised, $\|\phi_k\|_{\ltwomu}=1$ for all $k=1,\,\dots,\,K$. Ortogonality implies that all off-diagonal entries in $\mathcal{A}_K$ are zero. Hence $\mathcal{A}_K = Id$ and $\mathcal{A}_K^{-1/2}=Id$. Plugging these estimates in the bound obtained in Theorem \ref{th:backward} leads to the statement of the Corollary.
\end{proof}

%

\subsection{Forward Evaluation}
\label{forwardeval}

Note that the value iteration procedure described above provides not only an approximation of the value function but also an approximation of the optimal policy; in order to find the control at time $n$ and in state $x$, it is sufficient to solve the optimization problem in the last line of \eqref{eq:v_iter}. In practical applications, it is often the control policy not only the value that is of interest. In this section, we therefore assess the value (performance) of the estimated policy. The only output of the backward procedure we use is the matrix of projection coefficients $\{\hat{\alpha}_k^n\}_{n,k=1}^{N,K}$ which we employ  in a \emph{forward scheme} to take decisions. Recall that those projection coefficients are functions of $\tcX_n$, but this dependence is supressed below for the sake of clarity of notation. For a matrix $\{\xi^m_n\}_{m,n=1}^{M', N}$ of {i.i.d.} $U(0,1)$ variables, and a fixed $x$, we perform a Monte Carlo simulation as follows:
\begin{equation}
\label{eq:eval}
\begin{cases}
X^m_0&=x,\\
X_{n+1}^m&=\varphi(n,X_n^m,\xi^m_n,\nu_n^m),\qquad m = 1, \ldots, M', \ n=0, \ldots, N-1,\\
v(x)&=\frac{1}{M'}\sum_{m=1}^{M'}\Bigl(\sum_{n=1}^N f(n,X^m_n,\nu^m_n)+g(X^m_N) \Bigr),
\end{cases}
\end{equation}
where the estimated optimal control $\nu^m_n$ is computed as:
\begin{equation}
\label{eq:controlpo}
\nu^m_n=\argmax_{u \in\domU}\Bigl\{ f(n, X_n^m,u)+\sum_{k=1}^K\hat{\alpha}^{n+1}_k\hat{\phi}^n_k(X^m_{n}, u)\Bigr\}.
\end{equation}
 In the following we use the notation ``Evaluate the policy'' to refer to the routine specified by equations \eqref{eq:eval}-\eqref{eq:controlpo}.
 
The above Monte Carlo evaluation of the policy approximates $\tilde V(0, x)$, where the valuation function $\tilde V$ is defined as follows:
\begin{equation}\label{eqn:valuation_fun}
\begin{cases}
&\tilde{V}(N,x)=g(x)\\ 
&\tilde{V}(n,x)=f(n,x,\hat{u}_n(x))+\mathbb{E}\big[\tilde{V}(n+1,X_{n+1})\big|X_n=x,u_n=\hat{u}_n(x)\big],
\end{cases}
\end{equation}
where
\[
\hat u_n (x) = \argmax_{u \in\domU}\Bigl\{ f(n, x,u)+\sum_{k=1}^K\hat{\alpha}^{n+1}_k\hat{\phi}^n_k(x, u)\Bigr\}.
\]
Indeed, it is easy to see that the above defined valuation function has a representation
\[
\tilde V(n, x) = \mathbb{E} \left[ \sum_{t=n}^N f\big(t, X_t, \hat u_t(X_t)\big) + g(X_N) \right],
\]
where
\[
\begin{cases}
X_n&=x,\\
X_{t+1}&=\varphi(n,X_t,\xi_t, \hat u_t(X_t)),\qquad t=n, \ldots, N-1.
\end{cases}
\]

Using \eqref{eqn:valuation_fun} has advantages for proving convergence over the above forward running representation or its Monte Carlo estimate \eqref{eq:eval}-\eqref{eq:controlpo}. 

\begin{theorem}
\label{th:forward}
Under Assumptions \ref{a:density}-\ref{a:semicont_hat_phi}, for $n = 0, 1, \ldots, N$,
\[
\begin{split}
\Big\| \tilde{V}(n,\tilde{X})-V(n,\tilde{X})\Big\|_{\ltwoeNn}&\le
\frac{2\bar{R}}{(\bar{R}-1)^2}\Big((N-n)\bar{R}^{N-n+1}-(N-n+1)\bar{R}^{N-n}+1\Big)\\
&\qquad\times \Big(\epsilon_K+\frac{\sqrt{K}}{\sqrt{M}}\Gamma \big\|  \mathcal{A}_K^{-1/2}\big\|_{2} \max\limits_{k=1,\,\dots,\,K}\|\phi_k\big\|_{\ltwomu}\Big).
\end{split}
\]
where $\epsilon_K=\max_{n=1, \ldots, N} \Big\|\Pi^n_{K}V(n,\tilde{X})-V(n,\tilde{X}) \Big\|_{\ltwomu}$.
\end{theorem}

\begin{remark}\label{rem:forward}
Comparing the estimates from Theorems \ref{th:backward}-\ref{th:forward}, we get
\[
\frac{\Big\| \tilde{V}(n,\tilde{X})-V(n,\tilde{X})\Big\|_{\ltwoeNn}}{\Big\| \hat{V}(n,\tilde{X})-V(n,\tilde{X})\Big\|_{\ltwoeNn}}
=
2(N-n) + \frac{(N-n) (\bar{R}-1) - \bar{R}^{N-n}+1}{(\bar{R}^{N-n} -1)(\bar{R} -1)} \approx 2(N-n).
\]
Therefore, the value function is estimated considerably better than the policy. The performance of the estimated policy deteriorates with the number of periods till the horizon $N$. Understandably, as the non-optimally controlled process drifts away from its optimal trajectory.
\end{remark}

\subsection{Performance iteration}
\label{s:performance}

Our performance iteration algorithm for regression Monte Carlo is inspired by  \citet{Longstaff2001} who provide an alternative to the value iteration method presented in \citet{Tsitsiklis2001} in the framework of optimal stopping problems. Our iterative procedure is based on the dynamic programming equation for the performance measure (instead of that for the value function):
\begin{equation}
\label{eq:lspolicy}
\begin{cases}
J(N,X_N)&=g(X_N)\\
J\big(n,(X_s,  u^*_s)_{s=n+1, \ldots, N}\big)&=f(n, X_n,u^*_n(X_n))+J\big(n+1,(X_s,  u^*_s)_{s=n+1, \ldots, N}\big),
\end{cases}
\end{equation}
where the control is given by
\begin{equation}
\label{eq:control}
u^*_n (x)=\argmax_{u \in\domU}\Bigl\{ f(n, x,u)+\mathbb{E}\bigl[J(n+1,(X_s, u^*_s)_{s=n+1, \ldots, N}) \big|X_n=x,u_n=u \bigr]\Bigr\},
\end{equation}
and, with an abuse of notation as the control is not defined at $N$, $(X_s,  u^*_s)_{s=n+1, \ldots, N}$ denotes the process $(X_s)_{s = n, \ldots, N}$ controlled by the control maps $(u^*_s)_{n, \ldots, N-1}$, i.e., 
\[
X_{s+1} = \varphi(s, X_s, \xi_s, u^*(X_s)), \qquad s = n, \ldots, N-1.
\]

The value function is recovered by conditioning the performance measure on $X_n = x$:
\begin{equation}
\label{eq:policyit}
V(n,x)=\mathbb{E}\Bigl[J\big(n,(X_s, u^*_s)_{s=n, \ldots, N}\big) |X_n=x\Bigr].
\end{equation}
The above conditioning can be viewed as a projection, which will be particularly useful when assessing $V(n, \tilde X)$ with $\tilde X \sim \mu$ as an element of $L^2_\mu$.
Looking at the update rule that characterises the performance iteration approach it can be immediately seen that its main advantage compared to the value iteration is that the error committed in the estimation of the conditional expectation in   \eqref{eq:control} it is not directly propagated to the following time step. We will further discuss this topic in Subsection \ref{s:performancediscussion}.

A direct implementation of equation \eqref{eq:lspolicy} allows to iterate over $J$'s rather than $V$'s but requires the computation of 
\[
J(n,(X_s, u^*_s)_{s=n, \ldots, N})=\sum_{s=n}^{N-1} f(s,X_s,u^*_s(X_s))+g(X_N)
\] 
after $u^*_n$ is established, which numerically means resimulating the path from time $n$ to the terminal time for each time-$n$ training point incurring an additional computational cost. For details see Algorithm \ref{alg:RLMC}.

\begin{algorithm}[tb]
 \caption{Regress-later Monte Carlo algorithm (RLMC) - performance iteration}\label{alg:RLMC}
 \textbf{input:} M, K, $\mu$, $\{\phi\}_{k=1}^K$
 \begin{algorithmic}[1]
\State Pre-compute the inverse of the matrix $\mathcal{A}_K$
 \State Generate i.i.d. training points $\{\tilde{X}_{N}^{m}\}_{m=1}^{M}$ accordingly to the distribution $\mu$.
\State Initialise the performance measure $\hat{J} (N, m) =g(\tilde{X}_{N}^{m}), \quad \forall m$
 \For{n=N-1 to 0}
\State  $\hat\alpha^{n+1}=\mathcal{A}_K^{-1}\frac{1}{M}\sum_{m=1}^{M}\Big[J(n+1,\tilde{X}_{n+1}^{m}) \phi(\tilde{X}_{n+1}^{m})\Big]$
 \State Generate a new layer of i.i.d. training points $\{\tilde{X}_{n}^{m}\}_{m=1}^{M}$ accordingly to the distribution $\mu$.
\State Compute the control mapping $\hat{u}_n (x) =\argmax_{u \in\domU}\big\{ f(j, x,u)+\sum_{k=1}^K\hat\alpha^{n+1}_k\hat{\phi}^n_k(x,u)\big\}$
\For{m=1 to M}
\State Generate $(\xi^m_n, \ldots, \xi^m_{N-1}) \sim U(0,1)$
\State $X_n = \tilde X^m_n$
\For{j=s to N-1}
\State \(X_{s+1}=\varphi\big(s,X_{s},\xi^m_s,\hat{u}_j(X_s^m)\big)\)
\EndFor
\State Set $\hat J(n,m)=\sum_{s=n}^{N-1} f\big(n,X_s, \hat{u}_s (X_s)\big)+g(X_N)$
\EndFor
\EndFor
\end{algorithmic}
\textbf{output:} $\{\hat{\alpha}_{k}^n\}_{n,k=1}^{N,K}$
 \end{algorithm}

In order to assess the convergence of the performance iteration algorithm, we need to extend the notation used in the previous sections. As with each training point $X^m_n$ we need to simulate a controlled path up to time $N$,  the space $\ltwomu$ needs to be replaced with
\[
\ltwoxin = \ltwomu \otimes L^2\big( (0,1)^{N-n}, \lambda^{N-n}\big), \qquad n < N,
\]
where $\lambda$ is the Lebesgue measure on $(0,1)$. The elements of this space will be denoted $(\tilde X, \xi^{n}_{n}, \ldots, \xi^{n}_{N-1})$, where $\xi$'s correspond to uniform random variables driving the dynamics of the controlled Markov process \eqref{eq:Xdynamics}. To streamline notation, we also set $L^2_{\xi, N} = \ltwomu$. The space $L^2_{M,N-n+1} = (\ltwomu)^{M(N-n+1)}$ which collects all the randomness involved in computation of $\hat{\pmb{\alpha}}^n$ in the value iteration case gets a counterpart $\ltwochin$ defined by induction as follows:
\[
L^2_{\chi, N} = \big(\ltwomu \big)^M
\]
as no path is generated at time $N$, and
\[
\ltwochin = \ltwochinp \otimes \big(\ltwoxin \big)^M.
\]
In parallel, we define arguments of functions in $\ltwochin$: $\tcY_{N} = (\tilde X^1_N \times \cdots \times \tilde X_N^M)$ and, for $n < N-1$
\[
\tcY_n = \tcY_{n+1} \times \Big( \tilde X^m_{n}, \xi^{n,m}_{n}, \ldots, \xi^{n,m}_{N-1}\Big)_{m=1}^M.
\]
Finally, we introduce two counterparts of $\ltwoen$. The first one to assess performance of strategies:
\begin{equation}\label{eqn:ltwofn}
\ltwoFn = \ltwochinp \otimes \ltwoxin
\end{equation}
because the output of the algorithm is the control strategy which must be assessed by applying it between time $n$ and $N$ with the initial value $X_n = \tilde X \sim \mu$ and the remaining randomness used to obtain the trajectory until time $N$. 
The second counterpart of $\ltwoen$ is a subspace of $\ltwoFn$ which is used in assessing an estimated value function and regression coefficients $\hat{\pmb{\alpha}}^{n}$:
\[
\ltwofn = \ltwochinp \otimes \ltwomu.
\]

In the value iteration case, the estimated control $\hat u_n$ depends on all the training points at future times, i.e., on $\tcX_{n+1}$. Here, these controls involve  further random variables associated with simulation of the trajectory starting at every training point, which is indicated in $\tcY_{n+1}$. With an abuse of notation, we will write
\begin{equation}
\label{eq:perf_notation1}
(X_s, \hat u_s)_{s=n, \ldots, N}
\end{equation}
to mean the sequence of random variables dependent on $\tcY_{n+1}$ in the following way
\begin{equation}
\label{eq:perf_notation2}
X_{s+1} = \varphi\big(s, X_s, \xi_s, \hat u_s (\tcY_{s+1}; X_s)\big)
\end{equation}
with $X_n$ given and $(\xi_s)_{s=n}^{N-1}$ a sequence of i.i.d $U(0,1)$ random variables independent from $\tcY_{n+1}$. Therefore, $J\big(n, (X_s, \hat u_s)_{s=n, \ldots, N}) \big) \in \ltwoFn$, where $(X_n, \xi_n, \ldots, \xi_{N-1})$ are variables corresponding to the space $\ltwoxin$.
Notice that $J\big(n, (X_s, \hat u_s)_{s=n, \ldots, N}) \big)$ is a pathwise evaluation of the control policy, and depends therefore on $\tcY_{n+1}$ only through $\hat{u}$; this is in contraposition with the value iteration case, where the error propagates in time  also directly through the value function approximation $\hat{V}(n+1,\cdot)$.

We extend the projection operator as follows: for $h (\tcY_{n+1}; x, \xi_n, \ldots, \xi_{N-1}) \in \ltwoFn$ we set
\[
\Pi^{n}_K h (\tcY_{n+1}; x, \xi_n, \ldots, \xi_{N-1}) := \sum_{k=1}^K \alpha^{n}_k \phi_k(x),
\]
where
\[
\ltwochinp \ni \hat{\pmb{\alpha}}^{n} = \hat{\pmb{\alpha}}^{n} (\tcY_{n+1}) = \mathcal{A}_K^{-1} \mathbb{E}_{\xi, n} \big[ h (\tcY_{n+1}; \tilde X, \xi_n, \ldots, \xi_{N-1}) \pmb{\phi} (\tilde{X})\big]
\]
and $\mathbb{E}_{\xi, n}$ is the expectation linked to the space $\ltwoxin$. This is an ortogonal projection in $\ltwoxin$ on the space $lin \big(\phi_1(\tilde X), \ldots, \phi_K(\tilde X) \big)$, where $\tilde X$ is the variable corresponding to the $\ltwomu$ part of $\ltwoxin$.

The Monte Carlo projection operator is defined as follows. For a sequence of i.i.d. random variables $(\tilde X^1_{n}, \ldots, \tilde X^M_{n})$ and i.i.d. $(\xi^{n,m}_j)_{j = n, \ldots, N-1;\ m = 1, \ldots, M} \sim U(0,1)$, we set
\begin{equation}\label{eqn:perf_MCproj}
\hat \Pi^{n}_K h (\tcY_{n+1}; x, \xi_n, \ldots, \xi_{N-1}) := \sum_{k=1}^K \hat \alpha^{n}_k \phi_k(x),
\end{equation}
where
\begin{equation}
\ltwochin \ni \hat{\pmb{\alpha}}^{n} = \hat{\pmb{\alpha}}^{n}(\tcY_{n}) = \mathcal{A}_K^{-1} \frac1M \sum_{m=1}^M  h (\tcY_{n+1}; \tilde X^m_{n}, \xi^{n,m}_n, \ldots, \xi^{n,m}_{N-1}) \pmb{\phi} (\tilde{X}^m_n).
\end{equation}
Notice that $\hat \Pi^{n}_K h (\tcY_{n+1}; \cdot) \in \ltwofnl$.

We introduce now the extension of Lemma \ref{l:proj_error} and \ref{l:proj_bound} for functions living in the  spaces relevant for the performance iteration procedure. Proofs of these lemmas are a straightforward generalisation of those in Section \ref{s:assumptions} and are omitted.

\begin{lemma}[Projection error]
\label{l:proj_error_performance}
For $h\in \ltwoFn$, the error of the random projection operator is bounded as follows:
\begin{multline*}
\Big\|\hat{\Pi}^{N-n}_K h (\tcY_{n+1}; \cdot) -\Pi^{N-n}_K h (\tcY_{n+1}; \cdot) \Big\|_{\ltwofnl}\\
\le \big\|\mathcal{A}_K^{-1/2}\big\|_2\frac{1}{\sqrt{M}} SDev_{\substack{\tilde{X}\sim\mu,\\ \xi_{n},\,\dots,\,\xi_{N-1}\sim \lambda}} \Big(h(\tcY_{n+1}\,;\tilde{X}, \xi_n, \ldots, \xi_{N-1})\pmb{\phi}(\tilde{X})\Big),
\end{multline*}
where
\begin{multline*}
SDev_{\substack{\tilde{X}\sim\mu,\\ \xi_{n},\,\dots,\,\xi_{N-1}\sim \lambda}} \Big(h(\tcY_{n+1}\,;\tilde{X}, \xi_n, \ldots, \xi_{N-1})\pmb{\phi}(\tilde{X})\Big) \\
:=
\left( \sum_{k=1}^K Var_{\substack{\tilde{X}\sim\mu,\\ \xi_{n},\,\dots,\,\xi_{N-1}\sim\lambda}}\Big(h(\tcY_{n+1}\,;\tilde{X}, \xi_n, \ldots, \xi_{N-1})\phi_k(\tilde{X})\Big) \right)^{1/2},
\end{multline*}
and $\| \mathcal{A}_K^{-1/2} \|_2 = \max \{ \|\mathcal{A}_K^{-1/2} x\|_2:\ x \in \mathbb{R}^K, \ \|x\|_2 = 1 \}$ is the matrix operator norm of $\mathcal{A}_K^{-1/2}$.
\end{lemma}

\begin{lemma}[Bound on conditional expectation]
\label{l:traj_performance}
For any $h\in \ltwofn$, we have the following bound on the norm of the conditional expectation
\[
\big\| \sup_{u \in \mathcal{D}^U} \mathbb{E}\big[h (\tcY_{n+1}\,;X_{n+1})\big|X_n=\tilde{X}, u_n = u \big]\big\|_{\ltwofn}
\le\bar{R} \| h\|_{\ltwofn},
\]
where 
\[
\mathbb{E}\big[h (\tcY_{n+1}; X_{n+1})\big|X_n=x, u_n = u \big] = \int_{\mathcal{D}} h (\tcY_{n+1}; y)\; r(n, x, u; y) \mu(dy). 
\]
\end{lemma}

Consider the exact performance of the estimated optimal strategy $\hat u_n$ computed in Algorithm \ref{alg:RLMC}:
\[
\begin{split}
\tilde{V}(n,x)=&\mathbb{E}\Big[  J\big(n,(X_s,\hat{u}_s)_{s=n}^N\big) \Big|X_n=x \Big]\\
=&\mathbb{E}\Big[\sum_{s=n}^N f(s,X_s,\hat{u}_n(X_s))+g(X_N)\Big|X_n=x \Big]\\
=&f(n,x,\hat{u}_n(x))+\mathbb{E}\Big[\tilde{V}(n+1,X_{n+1})\Big|X_n=x ,u_n=\hat{u}_n(x)\Big].
\end{split}
\]
This is an analogous quantity as studied in the previous section concerned with the forward evaluation of a strategy extracted in the value iteration scheme. Note that $\tilde V(n, \cdot) \in \ltwofn$ due to the randomness used in computing the strategy $\hat u_n$.

 


\begin{theorem}
\label{th:performance}
Under Assumptions \ref{a:density}-\ref{a:semicont_hat_phi}, for $n = 0, 1, \ldots, N$,
\[
\Big\|\tilde{V}(n,\tilde{X})-V(n,\tilde{X})\Big\|_{\ltwofn}\le2\bar{R}\frac{(3\bar{R})^{N-n}-1}{3\bar{R}-1}\Big(\epsilon_K+\frac{\sqrt{K}}{\sqrt{M}}\Gamma \big\|  \mathcal{A}_K^{-1/2}\big\|_{2} \max\limits_{k=1,\,\dots,\,K}\|\phi_k\|_{\ltwomu}\Big),
\]
where $\epsilon_K=\max_{n=1, \ldots, N} \Big\|\Pi^n_{K}V(n,\tilde{X})-V(n,\tilde{X}) \Big\|_{\ltwomu}$.
\end{theorem}
\begin{corollary}
\label{c:performance}
Under the assumption of ortonormality of the basis functions $\phi_1, \ldots, \phi_K$ in $\ltwomu$, we have:
\[
\Big\|\tilde{V}(n,\tilde{X})-V(n,\tilde{X})\Big\|_{\ltwofn}\le2\bar{R}\frac{(3\bar{R})^{N-n}-1}{3\bar{R}-1}\Big(\epsilon_K+\frac{\sqrt{K}}{\sqrt{M}}\Gamma\Big).
\]
\end{corollary}
\begin{proof}
Analogous to the proof of Corollary \ref{c:backward}.
\end{proof}

\begin{remark}
We decided to use the name \emph{performance iteration} as opposed to \emph{policy iteration}, in order to avoid misunderstandings. Often in regression Monte Carlo literature on optimal stopping the analogue of the above algorithm is called policy iteration even though such a name is used in the more general approximate dynamic programming literature to describe algorithms which iterate over controls rather than over the performance measure.
\end{remark}

\subsection{Value vs. Performance iteration}
\label{s:performancediscussion}

In this section we compare the two iterative approaches presented in Section \ref{s:value} and \ref{s:performance}. First we comment on consequences of our theoretical results and then on our experience from solutions of practical problems.

 \subsubsection{Theoretical convergence}
 At a first glance, observing the error bound for the value and performance iteration algorithms (provided in theorems \ref{th:backward} and \ref{th:performance}) one might be tempted to claim that the former has, at least in general situations, a tighter error bound than the latter. Recall, however, that the quantities estimated by the two algorithms and assessed in the theorems are somewhat different. Value iteration provides an estimation of the value function (and this error is assessed in Theorem \ref{th:backward}), while the performance iteration provides an estimation of the control policy and it is its performance that is estimated in Theorem \ref{th:performance}. When comparing the quality of the estimated policies (which are of interest in most practical applications), we have to turn our attention to Theorem \ref{th:forward} which indeed provides us with error bounds for the performance of the control policy estimated by the value iteration algorithm. In Remark \ref{rem:forward} we have shown that this error is approximately $2N$ times higher than the error of the estimated value function.
Denoting by $\eta_V$ and $\eta_P$ the error bounds presented in Theorem \ref{th:forward} and \ref{th:performance}, a comparison of the error of the performance of the policies estimated by the value iteration and performance iteration algorithms gives (under assumption that $\bar{R}, N \gg 1$):
\begin{equation}
\label{eq:comparison_vp}
\frac{\eta_{V}}{\eta_{P}}\approx \frac{N}{3^{N-1}}.
\end{equation}

Define now $\tilde{\epsilon}_K=\max_{n=1, \ldots, N} \big\|\Pi^n_{K}\tilde{V}(n,\tilde{X})-\tilde{V}(n,\tilde{X}) \big\|_{\ltwomu}$, the projection error for the performance of estimated policy in performance iteration algorithm.
Assume that $\tilde{\epsilon}_K\approx \epsilon_K$, which hold, for example, when $\tilde{V}$ is a good approximate of $V$, then the error bound in Theorem \ref{th:performance} can be shrunk considerably:
 \begin{proposition}
 \label{p:performance}
 Under Assumptions \ref{a:density} and \ref{a:normV}, and further assuming $\tilde{\epsilon}_K\approx \epsilon_K$ we have  for $n = 0, 1, \ldots, N$
 \[
 \Big\|\tilde{V}(n,\tilde{X})-V(n,\tilde{X})\Big\|_{\ltwofn}\le2\bar{R}\frac{\bar{R}^{N-n}-1}{\bar{R}-1}\Big(\epsilon_K+\frac{\sqrt{K}}{\sqrt{M}}\Gamma \big\|  \mathcal{A}_K^{-1/2}\big\|_{2} \max\limits_{k=1,\,\dots,\,K}\|\phi_k\|_{\ltwomu}\Big).
 \]
 \end{proposition}
 We omit the proof as it follows easily from the proof of Theorem \ref{th:performance} and \ref{th:backward}.

Notice that under the assumptions of Proposition \ref{p:performance} the ratio in \eqref{eq:comparison_vp} becomes
\begin{equation}
\label{eq:comparison_vp2}
\frac{\eta_{V}}{\eta_{P}}\approx N,
\end{equation}
i.e., the performance of the policy estimated by the value iteration algorithm differs from the optimal value $N$ times more than the performance of the policy estimated by the performance iteration algorithm.


The above results give the comparison of the worst case estimates computed in Theorems \ref{th:backward}-\ref{th:performance}. In the following subsection we share our practical experience with those two algorithms.

 \subsubsection{Practical considerations}

In optimal stopping problems, the analogues of performance and value iterations algorithms have comparable computational complexities. In the context of controlled Markov processes these two algorithms have different computational complexities due to the additional evaluation step employed in performance iteration. Therefore, there is a more delicate decision to be made between using performance iteration or value iteration but with a larger number of training points and/or basis functions.

The main reason to decide to use performance iteration over value iteration is that the latter induces propagation of the projection error, while the former uses an update which does not depend \emph{directly} on the functional form of the estimator of the conditional expectation and, therefore, does not propagate the error.

\paragraph{Projection error.}
In practice the following conditions indicate increased effect of error propagation in the value iteration algorithm compared to the performance iteration algorithm:
\begin{itemize}
\item a value function that cannot be represented accurately using chosen basis functions,
\item a value function that abruptly change shape;
\item a running reward $f$ that is small compared to the conditional expectation term in the dynamic programming equation, and, therefore, is hardly represented in the basis function approximation;
\item basis functions can induce a ``sensible'' policy. 
\end{itemize}
In these situations we can observe a substantial improvement in the performance of the estimated control policy when using performance over value iteration. 
The greater precision is due to the ability of the former to produce regression coefficients that can adapt to small changes in the value function, captured by the evaluation step. Value iteration on the other hand reuses the estimated future conditional expectation to compute the current value function, effectively being blind to small changes in the true value function produced by the distribution on the controlled process at future time step. These small contributions can build up to a considerable error. Numerical evidence of the claims above can be found in Section \ref{sec:numtest}.

\paragraph{Variance.}
In practical runs, we observe fairly stable behaviour of projection coefficients $\hat{\pmb{\alpha}}^n$ across time for value iteration, while the same quantity for performance iteration exhibits substantial variations. This is due to the significantly larger variance of random projection of performance of the whole trajectories of controlled process (i.e., $\sigma(X_{n+1}, \ldots, X_N)$-measurable quantities) compared to the value iteration where one projects only $\sigma(X_{n+1})$-measurable variables. 

 
On the other hand, in the case of value iteration, the grater sensitivity to the approximation of the terminal condition, i.e., the first random projection executed in the algorithm, results in a greater variance of the quality of the estimated policy from run to run. The estimated value function at time $n$ is indeed based on the estimate at time $n+1$ plus the effect of the running profit $f$. For this reason problems where $f$ is small, compared to the terminal condition, are hugely influenced by the first regression approximation at time $N$. 

\medskip
To conclude notice that the performance iteration algorithm \ref{alg:RLMC} does not require any truncation of the estimates of the value function as it is based on the realised performance of a control policy. On the other hand, in the  value iteration algorithm we need to introduce truncation to alleviate effects of large projection errors and prevent further propagation of those. 

Some numerical examples of the claims above are presented in Section \ref{sec:numerics} where we study a toy problem whose characteristics highlight the differences between the two methods.

\subsection{Training measure and choice of basis functions}
\label{s:training_and_basis}
 In this section we analyse different training measures $\mu$ and we guide the reader on the choice of basis functions, indicating their pros and cons.
 
  \subsubsection{Choice of training measure $\mu$}
 As shown by Theorems \ref{th:backward}, \ref{th:forward} and \ref{th:performance} the choice of the training measure $\mu$ is paramount for a quick convergence. The distribution of the training points influences the quality of the estimations mainly through the bound on the transition density $\bar{R}$ and the representation error $\epsilon_K$.
 
 In the following we will present the consequences of choosing a particular measure $\mu$ in two common situations.
 
\paragraph{Uninformed choice.} When no information about the problem is exploited and the state space $\mathcal{D}$ is compact, we can use  a uniform distribution $\mu=\lambda / \lambda(\mathcal{D})$, where $\lambda$ is the Lebesque measure. The consequences are: a less accurate fit of basis functions (see Figure \ref{f:comparison_mu2}) and often an inflated error $\epsilon_K$.
 
\paragraph{Knowledge based.} When previous knowledge about the problem is available, the training measure $\mu$ can be chosen in order to maximise the quality of the control policy. Even though providing theoretical results in this direction is beyond the scope of this paper, we note that concentrated measures $\mu$ can produce very low values of $\epsilon_K$ (see Figure \ref{f:comparison_mu2}) at the cost of high values of $\bar{R}$, which reflects the poor coverage of the state space, i.e., insufficient number of training points in extreme positions to evaluate performance of ``wrong'' controls.
 
\medskip
We present now a practical example  for which we can compute values of $\bar{R}$ and $\epsilon_K$ and show the trade off between exploration and accuracy that arises when choosing the training measure $\mu$.
 Consider a two period model with the dynamics $X_1= (X_0 +u_0+\xi_0) \vee (-5) \wedge 5 $, where $\xi_0\sim\mathcal{N}(0,1)$ and the stochastic control problem
 \[
 V(0,z)=\inf_{u\in\mathcal{U}}\Big\{\mathbb{E}\big[ \frac{u_0^2}{2} +g(X_1)|X_0=z, u_0 = u\big] \Big\},
 \]
 where $g(x)=\min(1,x^2)$. We choose basis functions $\{1,\,x,\,x^2\}$ and a family of measures $\mu_{\sigma}=\mathcal{N}(0,\sigma)$. In this framework the trade off between $\epsilon_3$ and $\bar{R}$ is driven by the parameter $\sigma$, which determines the width of the training distribution. The effect of $\mu$ on the quality of the estimated conditional expectations can be assessed from Figure \ref{f:comparison_mu2} which displays, on the right, the actual values of $\epsilon_3$ and $\bar{R}$ for some choices of $\sigma$. The effect on the control policy, however, is more subtle. Figure \ref{f:comparison_mu2} displays on the left an example of the effect of $\sigma$ on the estimated control.

 \begin{figure}
\centering
\includegraphics[width=0.29\linewidth]{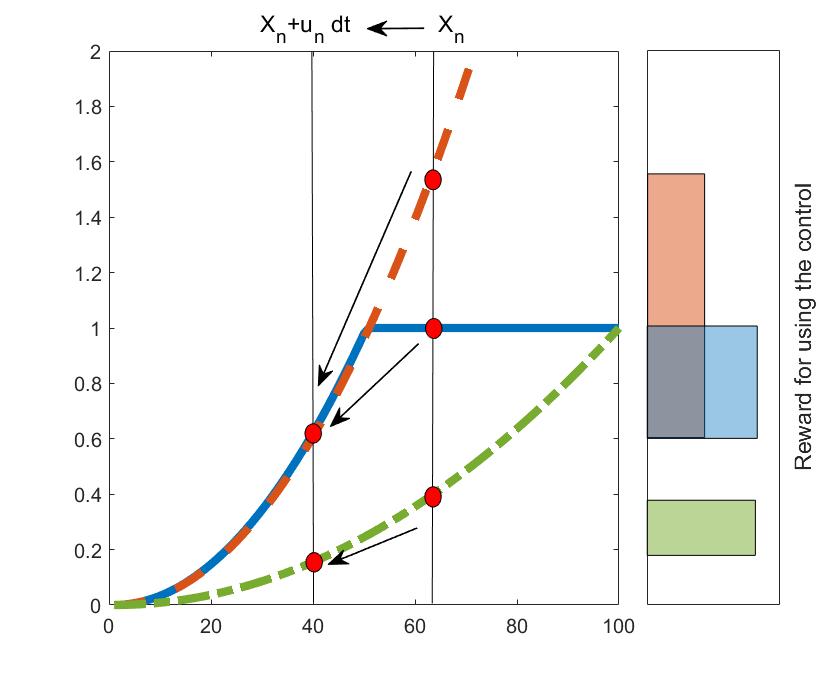}
\hspace{-5pt}\includegraphics[width=0.71\linewidth]{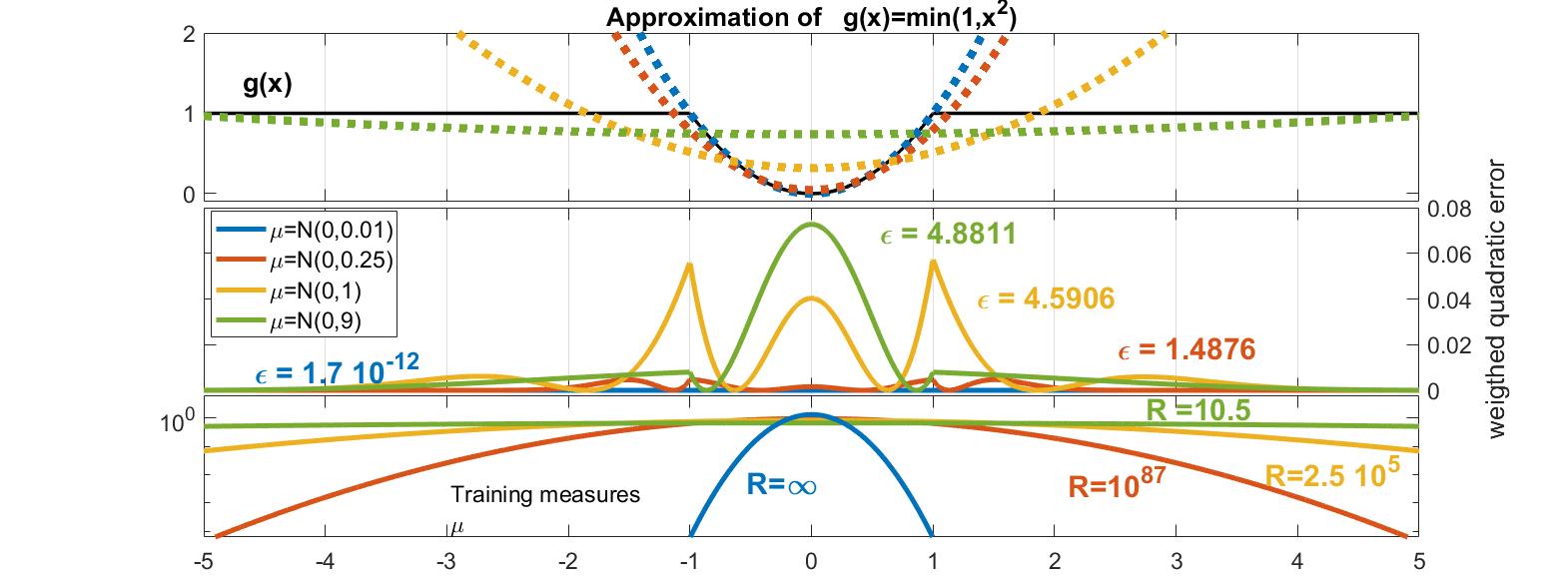}
\caption{In the figure above we show an example of the impact of the measure $\mu$ on $\bar{R}$ and  $\epsilon_3$. Color coding is identical on all graphs and explained in the legend. On the right, the top panel shows four different projections of $g(x)=\min(1,x^2)$ over $\{1,x,x^2\}$ produced by different measures $\mu$. The central panel displays the pointwise quadratic deviation of the projection from $g(x)$ and its integral $\epsilon_3$. In the lower panel we show the logplot of the density function of $\mu$ and corresponding $\bar R$; the value $\infty$ corresponds to the case in which $\bar R$ exceeds the range of numbers available in Matlab. Notice that more concentrated measures correspond to lower $\epsilon_3$ but higher $\bar{R}$, vice versa, higher $\epsilon_K$ and lower $\bar{R}$ are associated with more spread-out measures. We present on the left the same trade-off acting on the estimation of the control policy. The graph displays the true (blue) and estimated (green, orange) expected reward from choosing a control $u_n$. More concentrated measures overestimate the effect of the control and tend to be myopic, achieving smaller terminal payoff but paying a higher cost for control. Measures that allow for exploration are more conservative, tend to underestimate the value of a given action, and therefore the control is used only when it is very needed it.} 
\label{f:comparison_mu2}
\end{figure}

 \subsubsection{Choice of basis functions}
 In this section we discuss the choice of basis functions, which directly affect the precision of the estimates through $\epsilon_K$.
 
One of the most popular choices of basis functions when tackling a general stochastic control problem using regression Monte Carlo algorithms are monomials up to order $\kappa$. We shall show in the following the consequences of such a choice and examine some alternatives.

\paragraph{Monomials.} Notice that an arbitrary number of monomials , i.e. $\phi_k(x)=\prod_{s=1}^d x_s^{j_s}$, $j_s\ge0$ and $x=(x_1,\,\dots,\,x_d)$, are never orthogonal on any domain and under any training measure. As we cannot provide more specific error bounds than those in Theorem \ref{th:backward} and \ref{th:performance} in a general case, we choose a particular example often encountered in practice: we assume the domain $\mathcal{D}=[0,1]$ and the training measure $\mu=\lambda$ (a uniform on $\mathcal{D}$). In this scenario the matrix $\mathcal{A_K}$ is the Hilbert matrix $H_{i,j}=\frac{1}{i+j-2},\,i,j=1,\,\dots,\,K$, while the norm of the basis functions is bounded by $1$. In order to obtain an explicit error bound we have to assess $\| \mathcal{A}^{-1/2}_K \|_{2}$ for which we have an upper bound $1/\sqrt{\theta_K}$, where $\theta_K$ is a lower bound for the smallest eigenvalue of $\mathcal{A}_K$. The actual analytical expression of $\theta_{K}$ is complicated , but the asymptotic behaviour of $\| \mathcal{A}^{-1/2}_K \|_{2}$ is $\mathcal{O}\big( \sqrt{2^{4K^2-3K+1.5}(\pi K)^{4K^2-K}}  \big)=\mathcal{O}\big(K^{2K^2}  \big)$. It is well known that the Hilbert matrix is very ill conditioned and difficult to invert numerically affecting the accuracy of the algorithm. The final comment about monomials is that when the value function is close to a polynomial of a small degree, we have $\epsilon_K \approx 0$ without having to take $K$ large and the this choice of basis functions is useful. 

\paragraph{Orthonormal basis of polynomials.} A straightforward generalisation of the family of monomials  is given by orthogonal polynomial bases of which many examples exists on both compact and unbounded domains. In the literature, orthogonal polynomial bases have been extensively studied in the American option pricing framework and in relation to optimal stopping regression Monte Carlo algorithms, see \citet{Moreno2003} among others. Through our Theorems \ref{th:backward} and \ref{th:performance} and Corollaries \ref{c:backward} and \ref{c:performance} we can assess to which extent it is preferable to chose orthonormal polynomial functions over monomials. Notice also that for low order polynomials it is often possible to compute optimal controls in closed form in terms of the projection coefficients.

\paragraph{Locally affine approximation.} It is sometimes useful, when little is known about the structure of the value function, to exploit the flexibility of local approximations. A very popular choice is to take affine functions  with disjoint supports. Consider a partition $\mathcal{H}_1, \ldots, \mathcal{H}_I$ of the domain $\mathcal{D}$, usually consisting of hypercubes. For each hypercube $\mathcal{H}_i$, we take $d+1$ basis functions: $\mathds{1}_{\{x\in\mathcal{H}_i\}}, x_1\mathds{1}_{\{x\in\mathcal{H}_i\}}, \ldots,\allowbreak x_d \mathds{1}_{\{x\in\mathcal{H}_i\}}$. By construction, basis functions corresponding to different hypercubes are orthogonal and, under the Lebesgue training measure and hypercube partition, it is easy to compute $\mathcal{A}_K$ and its inverse analytically. 
Notice that in practical implementations this choice of basis functions has further advantages as the projection can be performed separately on each set $\mathcal{H}_i$ allowing for parallelisation, which eases the consequences of having a large number of basis functions. It should, however, be remarked that the evaluation of $\mathbb{E}_{n,x,u}[ X_{n+1,j} \mathds{1}_{\{X_{n+1}\in\mathcal{H}_{i}\}}]$ may be time consuming, even in Gaussian models, where multiple evaluations of standard normal CDF are needed.

\paragraph{Radial basis functions.} A less popular, but certainly interesting choice of basis functions when the domain $\mathcal{D}$ is high dimensional, is given by radial basis functions. This class of functions enjoys the property that $\phi_k(x-c_k)=\phi_k(\|x-c_k\|)$, where $c_k$ is called a centre, making them well suited for multidimensional settings. Here we assume that the norm is a weighted Euclidean norm (with a vector of weights $w_k$) corresponding to $\phi$ being constant on elipsoids with centre $c_k$ and decay rates (or bandwidths) $w_k$. Hence, after having identified the most suitable class of radial functions, we need to choose the centres and the weights as well as the number of those for the problem at hand. A common choice is a (truncated) Gaussian family, represented by the kernel functions 
$\phi_k(x)=\frac{\sqrt{w_{k,1} \cdots w_{k,d}}}{(2\pi)^{d/2}} \exp\big(\sum_{i=1}^d w_{k,i} (x_i-c_{k,i})^2\big)$. The truncation is due to a compact domain $\mathcal{D}$ on which the dominating measure $\mu$ is that corresponding to the uniform distribution. Such system of functions cannot be orthogonal, however, recalling that the Gaussian density function is almost zero in the tails, one should expect virtually zero entries in the matrix $\mathcal{A}_K$ apart from neighbouring functions. This will not only simplify the inversion of $\mathcal{A}_K$, but also yield a fairly modest norm $\|\mathcal{A}_K^{-1/2}\|_2$.

\section{Proofs}
\label{proofs}

In the proofs we will use the following shorthand notation $\mathbb{E}_{u_n(x)}[\,\cdot\,]:=\mathbb{E}[\,\cdot\,|X_n=x, u_n=u_n(x)]$.

\subsection{Value Iteration}

\subsubsection{Backward Estimation}

In this subsection we study the backward estimation of the value function $\hat{V}$. For convenience, denote 
\[
\doublehat{V}(n,x)=f(n,x,\hat{u}_n(x))+\mathbb{E}_{\hat{u}_n(x)}\big[\hat{\Pi}_K\hat{V}(n+1,\cdot)\big],
\]
where $\hat u_n (x) = \hat u_n(\tcX_n; x)$ is the estimated optimal policy.

\begin{proof}[Proof of Theorem \ref{th:backward}]
Recall that $\hat{V}(n,\tilde{X})\in\ltwoeNn$ and further notice that
\begin{equation}
\label{eq:t0}
\begin{split}
&\Big\|\hat{V}(n,\tilde{X})-V(n,\tilde{X}) \Big\|_{\ltwoeNn}
=\Big\|\doublehat{V}(n,\tilde{X})\wedge \Gamma \vee (-\Gamma)-V(n,\tilde{X}) \Big\|_{\ltwoeNn}\le\Big\|\doublehat{V}(n,\tilde{X})-V(n,\tilde{X}) \Big\|_{\ltwoeNn},
\end{split}
\end{equation}
where in the last inequality we used that $|V(n,x)| \le \Gamma$ for all $n,x$. Given the definition of $\doublehat{V}$, we have
\[
\begin{split}
\doublehat{V}(n,\tilde{X})-V(n,\tilde{X})
&=f\big(n,\tilde{X},\hat{u}_n(\tilde{X})\big)+\mathbb{E}_{\hat{u}_n(\tilde{X})}\Big[\hat{\Pi}^{n+1}_K \hat{V}(n+1,X_{n+1}) \Big]\\
&\hspace{13pt}-f\big(n,\tilde{X},u^*_n(\tilde{X})\big)-\mathbb{E}_{u^*_{n}(\tilde{X})}\Big[V(n+1,X_{n+1})\Big]\\
&\ge \mathbb{E}_{u^*_n(\tilde{X})}\Big[\hat{\Pi}^{n+1}_K \hat{V}(n+1,X_{n+1}) -V(n+1,X_{n+1})\Big],
\end{split}
\]
where the inequality is given by the substitution of $\hat{u}_n(x)$, which realises the maximum in $\doublehat{V}$, with the true optimal control $u^*_n(x)$.
Similarly replacing  $u^*_n(x)$ by $\hat{u}_n(x)$ we obtain an upper bound
\[
\doublehat{V}(n,\tilde{X})-V(n,\tilde{X})\le  \mathbb{E}_{\hat{u}_n(\tilde{X})}\Big[\hat{\Pi}^{n+1}_K \hat{V}(n+1,X_{n+1}) -V(n+1,X_{n+1})\Big].
\]
Therefore, using Assumption \ref{a:density} and  Lemma \ref{l:traj}  we have the following bound:
\begin{equation}
\label{eq:t1}
\begin{split}
\Big\|\doublehat{V}(n,\tilde{X})-V(n,\tilde{X}) \Big\|_{\ltwoeNn}
&\le \Big\|\sup_{u\in\domU}\mathbb{E}_{n,\tilde X, u}\Big[\hat{\Pi}^{n+1}_K \hat{V}(n+1,X_{n+1}) -V(n+1,X_{n+1})\Big] \Big\|_{\ltwoeNn}\\
&\le \bar{R}\Big\|\hat{\Pi}^{n+1}_K \hat{V}(n+1,\tilde{X}) -V(n+1,\tilde{X}) \Big\|_{\ltwoeNn}.
\end{split}
\end{equation}
We split now the term $\Big\|\hat{\Pi}^{n+1}_K \hat{V}(n+1,\tilde{X}) -V(n+1,\tilde{X}) \Big\|_{\ltwoeNn}$ into three components:
\begin{equation}
\label{eq:t2}
\begin{split}
\Big\|\hat{\Pi}^{n+1}_K \hat{V}(n+1,\tilde{X}) -V(n+1,\tilde{X}) \Big\|_{\ltwoeNn} &\le \Big\|\hat{\Pi}^{n+1}_K \hat{V}(n+1,\tilde{X}) -\Pi^{n+1}_K \hat{V}(n+1,\tilde{X}) \Big\|_{\ltwoeNn}\\
&\hspace{13pt}+\Big\|\Pi^{n+1}_K \hat{V}(n+1,\tilde{X}) -\Pi_K^{n+1} V(n+1,\tilde{X}) \Big\|_{\ltwoeNn}\\
&\hspace{13pt}+\Big\|\Pi_K^{n+1} V(n+1,\tilde{X}) -V(n+1,\tilde{X}) \Big\|_{\ltwoeNn}.
\end{split}
\end{equation}
For the first term in   \eqref{eq:t2} we have, using Lemma \ref{l:proj_error} and \ref{l:STDEV} and the bound $\Gamma$ for $\hat V$:
\begin{equation}
\label{eq:backstd}
\begin{split}
&\Big\| \hat{\Pi}^{n+1}_K\hat{V}(n+1,\tilde X)-\Pi^{n+1}_K\hat{V}(n+1,\tilde X)\Big\|_{\ltwoeNn}\\
&= \Big\|\big\|  \hat{\Pi}^{n+1}_K\hat{V}(n+1,\tilde X)-\Pi^{n+1}_K\hat{V}(n+1,\tilde X) \big\|_{\ltwoe} \Big\|_{L^2_{M,N-n-1}}
\le \frac{\sqrt{K}}{\sqrt{M}}\Gamma \big\|  \mathcal{A}_K^{-1/2}\big\|_{2}  \max\limits_{k=1,\,\dots,\,K}\|\phi_k\big\|_{\ltwomu}.
\end{split}
\end{equation}
The second term in   \eqref{eq:t2} represents the backward propagation of the error and, using Remark \ref{l:proj_bound}, can be used to set up a recursive relation between errors at different time steps:
\[
\begin{split}
\Big\|\Pi^{n+1}_K \hat{V}(n+1,\tilde{X}) -\Pi_K^{n+1} V(n+1,\tilde{X}) \Big\|_{\ltwoeNn}
&=\Big\|\Pi^{n+1}_K \hat{V}(n+1,\tilde{X}) -\Pi_K^{n+1} V(n+1,\tilde{X}) \Big\|_{\ltwoeNnp}\\
&\le\Big\|\hat{V}(n+1,\tilde{X}) -V(n+1,\tilde{X}) \Big\|_{\ltwoeNnp}.
\end{split}
\]
The last term in equation \eqref{eq:t2} is bounded by $\epsilon_K$:
\[
\big\|\Pi_K^{n+1} V(n+1,\tilde{X}) -V(n+1,\tilde{X}) \big\|_{\ltwoeNn} = \big\|\Pi_K^{n+1} V(n+1,\tilde{X})-V(n+1,\tilde{X})\big\|_{\ltwomu} \le\epsilon_K.
\]

Collecting the above estimates we find that
\begin{equation}
\label{eq:t4}
\begin{split}
\Big\|\hat{\Pi}^{n+1}_K \hat{V}(n+1,\tilde{X}) -V(n+1,\tilde{X}) \Big\|_{\ltwoeNn}
&\le\Big\|\hat{V}(n+1,\tilde{X}) -V(n+1,\tilde{X}) \Big\|_{\ltwoeNnp}+\epsilon_K\\
&\qquad+\frac{\sqrt{K}}{\sqrt{M}}\Gamma \big\|  \mathcal{A}_K^{-1/2}\big\|_{2}  \max\limits_{k=1,\,\dots,\,K}\|\phi_k\|_{\ltwomu}.
\end{split}
\end{equation}
Combining this result with   \eqref{eq:t1} and \eqref{eq:t0}, and denoting $\beta_n := \Big\|\hat{V}(n,\tilde{X})-V(n,\tilde{X}) \Big\|_{\ltwoeNn}$, leads to the following recursion:
\begin{equation}\label{eqn:final_beta}
\begin{split}
\beta_n
&\le\bar{R}\beta_{n+1}+\bar{R}\Big(\epsilon_K+\frac{\sqrt{K}}{\sqrt{M}}\Gamma \big\|  \mathcal{A}_K^{-1/2}\big\|_{2}  \max\limits_{k=1,\,\dots,\,K}\|\phi_k\|_{\ltwomu}\Big)\\
&=\bar{R}\sum_{s=0}^{N-n-1}\bar{R}^s\Big(\epsilon_K+\frac{\sqrt{K}}{\sqrt{M}}\Gamma \big\|  \mathcal{A}_K^{-1/2}\big\|_{2}  \max\limits_{k=1,\,\dots,\,K}\|\phi_k\|_{\ltwomu}\Big)\\
&=\bar{R}\frac{\bar{R}^{N-n}-1}{\bar{R}-1}\Big(\epsilon_K+\frac{\sqrt{K}}{\sqrt{M}}\Gamma \big\|  \mathcal{A}_K^{-1/2}\big\|_{2}  \max\limits_{k=1,\,\dots,\,K}\|\phi_k\|_{\ltwomu}\Big),
\end{split}
\end{equation}
where we used that $\hat{V}(N,\cdot)=V(N,\cdot)$, i.e. $\beta_N=0$.
\end{proof}

\subsubsection{Forward Evaluation}

\begin{proof}[Proof of Theorem \ref{th:forward}]
The proof will follow by induction. Notice that $V(N,x)=\tilde{V}(N,x)$ for all $x\in \mathcal{D}$. For $n<N$ we have
\[
\begin{split}
\tilde{V}(n,x)-V(n,x)
&=f(n,x,\hat{u}_n(\tilde{X}))+\mathbb{E}_{\hat{u}_n(x)}[\tilde{V}(n+1,X_{n+1})]-V(n,x)\\
&\hspace{13pt}+\mathbb{E}_{\hat{u}_n(x)}[\hat{\Pi}^{n+1}_K \hat{V}(n+1,X_{n+1})]-\mathbb{E}_{\hat{u}_n(x)}[\hat{\Pi}^{n+1}_K \hat{V}(n+1,X_{n+1})]\\
&=\doublehat{V}(n,x)-V(n,x)+\mathbb{E}_{\hat{u}_n(x)}[\tilde{V}(n+1,X_{n+1})]-\mathbb{E}_{\hat{u}_n(x)}[\hat{\Pi}^{n+1}_K \hat{V}(n+1,X_{n+1})],
\end{split}
\]
where we added and subtracted $\mathbb{E}_{\hat{u}_n(\tilde{X})}\big[ \hat{\Pi}^{n+1}_K \hat{V}(n+1,X_{n+1}) \big]$ and used that $\doublehat{V}(n,x)=f(x,\hat{u}_n(x))+\mathbb{E}_{\hat{u}_n(x)}\big[\hat{\Pi}^{n+1}_K\hat{V}(n+1,X_{n+1})\big]$. Hence,
\begin{equation}
\label{eq:forwardsplit}
\begin{split}
\Big\| \tilde{V}(n,\tilde{X})-V(n,\tilde{X})\Big\|_{\ltwoeNn}
&\le\Big\| \doublehat{V}(n,\tilde{X})-V(n,\tilde{X})\Big\|_{\ltwoeNn}\\
&\qquad+\Big\|\mathbb{E}_{\hat{u}(\tilde{X})}\Big[\tilde{V}(n+1,X_{n+1})-\hat{\Pi}^{n+1}_{K} \hat{V}(n+1,X_{n+1})\Big]\Big\|_{\ltwoeNn}.
\end{split}
\end{equation}

The first term in  \eqref{eq:forwardsplit} has been bounded in the proof of Theorem \ref{th:backward}, Eq. \eqref{eqn:final_beta}, as follows
\begin{equation}
\label{eq:forward1}
\begin{split}
\Big\| \doublehat{V}(n,\tilde{X})-V(n,\tilde{X})\Big\|_{\ltwoeNn}
&\le\bar{R}\Big\|\hat{V}(n+1,\tilde{X}) -V(n+1,\tilde{X}) \Big\|_{\ltwoeNnp}+\bar{R}\epsilon_K\\
&\qquad+\bar{R}\frac{\sqrt{K}}{\sqrt{M}}\Gamma \big\|  \mathcal{A}_K^{-1/2}\big\|_{2}  \max\limits_{k=1,\,\dots,\,K}\|\phi_k\|_{\ltwomu}.
\end{split}
\end{equation}
The second term in   \eqref{eq:forwardsplit} can be bounded making use of Assumption \ref{a:density} and Lemma \ref{l:traj}:
\begin{equation}
\label{eq:forward2}
\begin{split}
&\Big\|\mathbb{E}_{\hat{u}_n(\tilde{X})}\Big[\tilde{V}(n+1,X_{n+1})-\hat{\Pi}^{n+1}_{K} \hat{V}(n+1,X_{n+1})\Big]\Big\|_{\ltwoeNn}\\
&\le\bar{R}\Big\|\tilde{V}(n+1,\tilde{X})-\hat{\Pi}^{n+1}_{K} \hat{V}(n+1,\tilde{X})\Big\|_{\ltwoeNn}\\
&\le\bar{R}\Big\|\tilde{V}(n+1,\tilde{X})-V(n+1,\tilde{X})\Big\|_{\ltwoeNnp}
+\bar{R}\Big\|V(n+1,\tilde{X})-\hat{\Pi}^{n+1}_{K} \hat{V}(n+1,\tilde{X})\Big\|_{\ltwoeNn},
\end{split}
\end{equation}
where the second inequality has been obtained using triangular inequality  with the aim of highlighting the term representing the propagation of the error.
The second term in   \eqref{eq:forward2} has been estimated in \eqref{eq:t4}:
\begin{equation}
\label{eq:forward5}
\begin{split}
\Big\|V(n+1,\tilde{X})-\hat{\Pi}_{K} \hat{V}(n+1,\tilde{X})\Big\|_{\ltwoeNn}
&\le \epsilon_K +\Big\| V(n+1,\tilde{X})-\hat{V}(n+1,\tilde{X})\Big\|_{\ltwoeNnp}\\
&\hspace{13pt} +\frac{\sqrt{K}}{\sqrt{M}}\Gamma \big\|  \mathcal{A}_K^{-1/2}\big\|_{2} \max\limits_{k=1,\,\dots,\,K}\|\phi_k\|_{\ltwomu}.
\end{split}
\end{equation}

Combining estimates  \eqref{eq:forward1}-\eqref{eq:forward5} we obtain 
\begin{equation}
\label{eq:forward6}
\begin{split}
\Big\| \tilde{V}(n,\tilde{X})-V(n,\tilde{X})\Big\|_{\ltwoeNn}
&\le \bar{R} \Big\| \tilde{V}(n+1,\tilde{X})-V(n+1,\tilde{X})\Big\|_{\ltwoeNnp}\\
&\hspace{13pt}+2\bar{R} Z\\
&\hspace{13pt}+2\bar{R}\Big\| V(n+1,\tilde{X})-\hat{V}(n+1,\tilde{X})\Big\|_{\ltwoeNnp},
\end{split}
\end{equation}
where
\[
Z :=\epsilon_K+\frac{\sqrt{K}}{\sqrt{M}}\Gamma \big\|  \mathcal{A}_K^{-1/2}\big\|_{2} \max\limits_{k=1,\,\dots,\,K}\|\phi_k\|_{\ltwomu}.
\]
From Theorem \ref{th:backward} we have
\[
\Big\| V(n+1,\tilde{X})-\hat{V}(n+1,\tilde{X})\Big\|_{\ltwoeNnp}\le \bar{R}\frac{\bar{R}^{N-n-1}-1}{\bar{R}-1} Z.
\]
Denote $\gamma_n:=\Big\| \tilde{V}(n,\tilde{X})-V(n,\tilde{X})\Big\|_{\ltwoeNn}$. From \eqref{eq:forward6} and the above estimate we obtain a recursion for $\gamma_n$:
 \begin{equation}
\label{eq:forward7}
\begin{split}
\gamma_n
&\le \bar{R}\gamma_{n+1}+2\bar{R}Z\Big(\frac{\bar{R}^{N-n}-\bar{R}}{\bar{R}-1}+1 \Big)\\
&=\frac{2\bar{R}}{(\bar{R}-1)^2}Z\Big((N-n)\bar{R}^{N-n+1}-(N-n+1)\bar{R}^{N-n}-1\Big),
\end{split}
\end{equation}
where we used that $\gamma_N = 0$ and $\hat V(N, x) = V(X, x)$. The statement of the theorem follows.
\end{proof}

\subsection{Performance Iteration}
\begin{proof}[Proof of Theorem \ref{th:performance}]
We have
\[
\begin{split}
&\tilde{V}(n,\tilde{X})-V(n,\tilde{X})\\
&=f(n,X_n,\hat{u}_n(\tilde{X}))+\mathbb{E}_{\hat{u}_{n}(\tilde{X})}\Big[\tilde{V}(n+1,X_{n+1})\Big]
-f(n,\tilde{X},u^*_n(\tilde{X}))-\mathbb{E}_{u^*_{n}(\tilde{X})}\Big[V(n+1,X_{n+1})\Big]\\
&\hspace{13pt}+\mathbb{E}_{\hat{u}_{n}(\tilde{X})}\Big[\hat{\Pi}^{n+1}_KJ\big(n+1,(X_s,\hat{u}_s)_{s=n+1}^N\big)\big) \Big]-\mathbb{E}_{\hat{u}_{n}(\tilde{X})}\Big[\hat{\Pi}^{n+1}_KJ\big(n+1,(X_s,\hat{u}_s)_{s=n+1}^N\big)\big) \Big],
\end{split}
\]
where we have added and subtracted $\mathbb{E}_{\hat{u}_{n}(\tilde{X})}\big[\hat{\Pi}^{n+1}_KJ\big(n+1, \ldots )\big]$. Recall that $\hat{\Pi}^{n+1}_K J(n+1,\ldots) \in \ltwofn$. Using the notation
$\mathbb{E}_{n, x, u}\big[\hat{\Pi}^{n+1}_KJ(n+1,\ldots)\big]$ we mean that the coordinate $\tilde X$ corresponding to $\ltwomu$ part of $\ltwofn$ equals $X_{n+1}$.
Since 
\[
\hat{u}_n(x)=\argmax_{u \in \domU} \big\{f(n,x,u)+\mathbb{E}_{n,x,u}\big[\hat{\Pi}^{n+1}_KJ\big(n+1,(X_s,\hat{u}_s)_{s=n+1}^N\big)\big]\big\},
\]
we get a lower bound:
\[
\begin{split}
\tilde{V}(n,\tilde{X})-V(n,\tilde{X})
&\ge \mathbb{E}_{\hat{u}_{n}(\tilde{X})}\Big[\tilde{V}(n+1,X_{n+1})-\hat{\Pi}^{n+1}_K J\big(n+1,(X_s,\hat{u}_s)_{s=n+1}^N\big)\Big]\\
&\qquad+\mathbb{E}_{u^*_{n}(\tilde{X})}\Big[\hat{\Pi}^{n+1}_KJ\big(n+1,(X_s,\hat{u}_s)_{s=n+1}^N\big)-V(n+1,X_{n+1})\Big].
\end{split}
\]
Similarly, since $u^*_n(x)=\argmax_{u \in \domU} \big\{f(n, x,u_n)+\mathbb{E}_{n, x, u}\big[V(n+1,X_{n+1})\big]\big\}$, we have:
\[
\begin{split}
\tilde{V}(n,\tilde{X})-V(n,\tilde{X})
&\le \mathbb{E}_{\hat{u}_{n}(\tilde{X})}\Big[\tilde{V}(n+1,X_{n+1})-\hat{\Pi}^{n+1}_KJ\big(n+1,(X_s,\hat{u}_s)_{s=n+1}^N\big)\Big]\\
&\hspace{12pt}+\mathbb{E}_{\hat{u}_{n}(\tilde{X})}\Big[\hat{\Pi}^{n+1}_KJ\big(n+1,(X_s,\hat{u}_s)_{s=n+1}^N\big)-V(n+1,X_{n+1})\Big].
\end{split}
\]

Collecting the previous two inequalities and using the triangular inequality, we obtain
\[\begin{split}
\Big\|\tilde{V}(n,\tilde{X})-V(n,\tilde{X})\Big\|_{\ltwofn}
&\le \Big\| \mathbb{E}_{\hat{u}_n(\tilde{X})}\big[ \tilde{V}(n+1,X_{n+1})-\hat{\Pi}^{n+1}_KJ\big(n+1,(X_s,\hat{u}_s)_{s=n+1}^N\big)   \big]\Big\|_{\ltwofn}\\
&\hspace{12pt}+\Big\| \sup_{u\in\domU}\Big\{ \mathbb{E}_{n, \tilde X, u}\big[ \hat{\Pi}^{n+1}_KJ\big(n+1,(X_s,\hat{u}_s)_{s=n+1}^N\big) -V(n+1,X_{n+1})\big]\Big\}\Big\|_{\ltwofn}.
\end{split}
\]
Using Assumption \ref{a:density} and Lemma \ref{l:traj_performance}, the above inequality reads 
\begin{equation}
\label{eq:performance2}
\begin{split}
\Big\|\tilde{V}(n,\tilde{X})-V(n,\tilde{X})\Big\|_{\ltwofn}
&\le \bar{R}\Big\|\tilde{V}(n+1,\tilde X)-\hat{\Pi}^{n+1}_KJ\big(n+1,(X_s,\hat{u}_s)_{s=n+1}^N\big) \Big\|_{\ltwofn}\\
&\hspace{12pt}+\bar{R} \Big\|  \hat{\Pi}^{n+1}_KJ\big(n+1,(X_s,\hat{u}_s)_{s=n+1}^N\big) -V(n+1,\tilde X)\Big\|_{\ltwofnp}\\
&\le\bar{R}\Big\|\tilde{V}(n+1,\tilde X)-V(n+1,\tilde X)  \Big\|_{\ltwofnp} \\
&\hspace{12pt}+\bar{R}\Big\|\hat{\Pi}^{n+1}_K J\big(n+1,(X_s,\hat{u}_s)_{s=n+1}^N\big)  - V(n+1,\tilde X) \Big\|_{\ltwofn}\\
&\hspace{12pt}+\bar{R}\Big\|V(n+1,\tilde X)-\hat{\Pi}^{n+1}_K J\big(n+1,(X_s,\hat{u}_s)_{s=n+1}^N\big)  \Big\|_{\ltwofn}\\
&\le\bar{R}\Big\|\tilde{V}(n+1,\tilde X)-V(n+1,\tilde X)  \Big\|_{\ltwofnp} \\
&\hspace{12pt}+2 \bar{R}\Big\|V(n+1,\tilde X)-\hat{\Pi}^{n+1}_K J\big(n+1,(X_s,\hat{u}_s)_{s=n+1}^N\big)  \Big\|_{\ltwofn}.
\end{split}
\end{equation}
The first term in the last bound above represents  the propagation of error from future time steps. For the second we use the triangular inequality in order to split it in a number of error terms:
\begin{equation}
\label{eq:performance3}
\begin{split}
\Big\|\hat{\Pi}^{n+1}_K J\big(n+1,&(X_s,\hat{u}_s)_{s=n+1}^N\big) - V(n+1,X_{n+1}) \Big\|_{\ltwofn}\\
&\le\Big\|\hat{\Pi}^{n+1}_K J\big(n+1,(X_s,\hat{u}_s)_{s=n+1}^N-\Pi^{n+1}_K J\big(n+1,(X_s,\hat{u}_s)_{s=n+1}^N\big)\big) \Big\|_{\ltwofn}\\
&\qquad+ \Big\|\Pi^{n+1}_KJ\big(n+1,(X_s,\hat{u}_s)_{s=n+1}^N\big) -\Pi^{n+1}_K \tilde{V}(n+1,\tilde{X})\Big\|_{\ltwofnp}\\
&\qquad+\Big\|\Pi^{n+1}_K \tilde{V}(n+1,\tilde{X})- \Pi^{n+1}_K V(n+1,\tilde{X}) \Big\|_{\ltwofnp}\\
&\qquad+ \Big\|\Pi^{n+1}_K V(n+1,\tilde{X})-V(n+1,\tilde X) \Big\|_{\ltwomu}.
\end{split}
\end{equation}
The first term above can be bounded using Lemma \ref{l:proj_error_performance} and \ref{l:STDEV}:
\begin{equation}
\label{eq:p_std}
\begin{split}
 \Big\|\Pi^{n+1}_K J\big(n+1,(X_s,\hat{u}_s)_{s=n+1}^N\big)-\hat{\Pi}^{n+1}_KJ\big(n+1,(X_s,\hat{u}_s)_{s=n+1}^N\big) \Big\|_{\ltwofn}
 \le \frac{\sqrt{K}}{\sqrt{M}}\Gamma \big\|  \mathcal{A}_K^{-1/2}\big\|_{2}  \max\limits_{k=1,\,\dots,\,K}\|\phi_k\big\|_{\ltwomu}.
\end{split}
\end{equation}
The second term  in \eqref{eq:performance3} can be computed as follow using Remark \ref{l:proj_bound}:
\begin{equation}
\label{eq:p_zero}
\begin{split}
\Big\|\Pi^{n+1}_K &\tilde{V}(n+1,\tilde{X})-\Pi^{n+1}_K J\big(n+1,(X_s,\hat{u}_s)_{s=n+1}^N\big) \Big\|_{\ltwofnp}\\
&=\Big\|\mathcal{A}_K^{-1}\mathbb{E}_{z\sim\mu}\Big[\tilde{V}(n+1,z)\pmb{\phi}(z)\Big] \pmb{\phi}(\tilde{X})\\
&\quad-\mathcal{A}_K^{-1}\mathbb{E}_{\substack{z\sim\mu,\\ \xi_{n+1},\,\dots,\,\xi_{N-1}\sim\lambda}}\Big[J\big(n+1,(\varphi(s,X_{s},\hat{u}_{s},\xi_{s}),\hat{u}_{s+1})_{s=n+1}^{N-1}|_{X_{n+1}=z}\big) \pmb{\phi}(z)\Big]\pmb{\phi}(\tilde{X}) \Big\|_{\ltwofnp}\\
&=\Big\|\mathcal{A}_K^{-1}\mathbb{E}_{z\sim\mu}\Big[
\Big\{ \tilde{V}(n+1,z)- \mathbb{E}_{\xi_{n+1},\,\dots,\,\xi_{N-1}\sim\lambda}\big[ J\big(n+1,(\varphi(s,X_{s},\hat{u}_{s},\xi_{s}),\hat{u}_{s+1})_{s=n+1}^{N-1}\big)\big|_{X_{n+1}=z}\Big\}\\
&\qquad\times\pmb{\phi}(z)\Big] \pmb{\phi}(\tilde{X}) \Big\|_{\ltwofnp}\\
&=\Big\|\mathcal{A}_K^{-1}\mathbb{E}_{z\sim\mu}\Big[\big(\tilde{V}(n+1,z)- \tilde{V}(n+1,z) \big)\pmb{\phi}(z)\Big]\pmb{\phi}(\tilde{X}) \Big\|_{\ltwofnp}=0,
\end{split}
\end{equation}
where we used the tower property of conditional expectations and the notation introduced in \eqref{eq:perf_notation1}-\eqref{eq:perf_notation2}. 
The third term in \eqref{eq:performance3} can be bounded making use of Remark \ref{l:proj_bound}
\begin{equation}
\label{eq:p_iter}
\Big\|\Pi^{n+1}_K V(n+1,\tilde{X}) - \Pi^{n+1}_K \tilde{V}(n+1,\tilde{X}) \Big\|_{\ltwofnp}\le\Big\|V(n+1,\tilde{X}) - \tilde{V}(n+1,\tilde{X}) \Big\|_{\ltwofnp},
\end{equation}
note that we obtain an additional term representing the error propagated from future time steps.
The last term  in \eqref{eq:performance3} can be bounded by $\epsilon_K$.

Collecting above estimates we obtain
\begin{equation}
\label{eq:performance5}
\begin{split}
\Big\|\tilde{V}(n,\tilde{X})-V(n,\tilde{X})\Big\|_{\ltwofn}
&\le 3\bar{R} \Big\| \tilde{V}(n+1,\tilde{X})-V(n+1,\tilde{X})  \Big\|_{\ltwofn}\\
&\quad+2\bar{R} \epsilon_K +2\bar{R}\frac{\sqrt{K}}{\sqrt{M}}{\Gamma} \big\|  \mathcal{A}_K^{-1/2}\big\|_{2}  \max\limits_{k=1,\,\dots,\,K}\|\phi_k\big\|_{\ltwomu}.
\end{split}
\end{equation}
Let $\beta_n:=\Big\| \tilde{V}(n,\tilde{X})-V(n,\tilde{X})\Big\|_{\ltwofn}$ and $Z:=\epsilon_K+\frac{\sqrt{K}}{\sqrt{M}}\bar{\Gamma}\big\|  \mathcal{A}_K^{-1/2}\big\|_{2} \max\limits_{k=1,\,\dots,\,K}\|\phi_k\big\|_{\ltwomu}$. Inequality \eqref{eq:performance5} provides the following recursion for $\beta_n$:
\[
\begin{split}
\beta_n
&=3\bar{R}\beta_{n+1}+2\bar{R} Z=2\bar{R}\sum_{s=0}^{N-n-1}(3\bar{R})^s Z=2\bar{R}\frac{(3\bar{R})^{N-n}-1}{3\bar{R}-1}Z,
\end{split}
\]
which provides us with the statement of the theorem.
\end{proof}

\section{Numerical Examples}
\label{sec:numerics}

\subsection{LQ1 convergence to analytical solution}
In order to provide evidence of the convergence of the two algorithms we briefly present a linear quadratic problem in one dimension, for which analytical solution is available in continuous time. The dynamics in continuous time is given by
\[
dX_t = (1+X_t+u_t)dt + dW_t,
\]
with the control $(u_t)$ being a real-valued process adapted to the filtration generated by the Brownian motion $(W_t)$.
Discretising the time with time-step $1/N$ yields a process with the dynamics:
\[
X_{n+1}=X_n+\frac{1+X_n+u_n}{N}+\frac{\xi_n}{\sqrt{N}}, \qquad X_0=x_0,\,\,\xi_n\sim\mathcal{N}(0,1)
\]
Define the cost functional $J$
\[
J\big(n,(X_s,u_s)_{s=n}^N\big)=\sum_{s=n}^{N}\frac{X_s^2+u_s^2}{N}+X_{N}^2
\]
and the value function $V(n,x)=\sup\limits_{(u_t)}\Big\{ J\big(n,(X_s,u_s)_{s=n}^{N}\big) \Big\}$.

We choose $N=100$ and solve this problem using the two algorithms presented in Section \ref{s:RLMC}. We compare the value of the estimated policies with the value function of the continuous time problem. The relative error is displayed in Figure \ref{fig:LQ1_error}; notice that we can expect about $1\%$ error coming from the discretisation of time. 

 \begin{figure}
\centering
\includegraphics[width=0.8\linewidth]{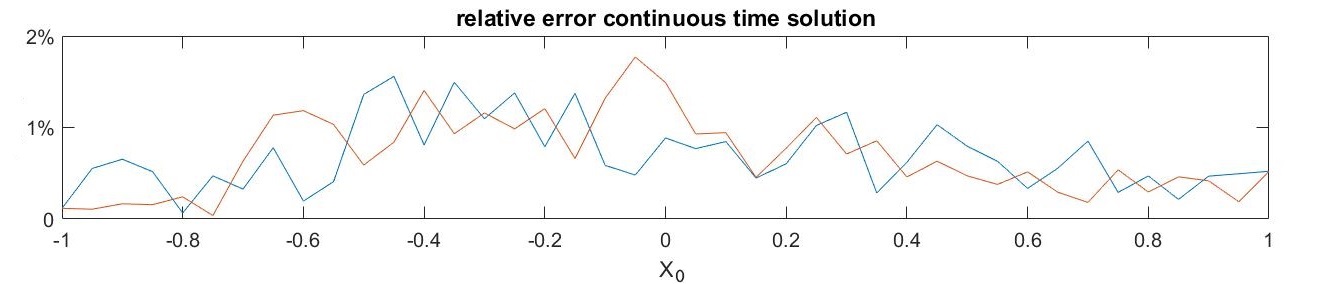}
\caption{In the figure above we show, in relative terms, the distance between the value function of the Linear Quadratic continuous time problem and the value of the policies estimated by the algorithms, value iteration in blue and performance iteration in orange, as a function of the initial condition $x_0$}
\label{fig:LQ1_error}
\end{figure}

\subsection{Control of a particle through doorways}
\label{sec:numtest}

In this experiment we propose a toy problem whose optimal policy is difficult to learn for the algorithm. 
To help intuition, imagine we are controlling a particle through a system of rooms connected via the doors $[d^-_{i},d^+_{i}]$ at times $t_1, \ldots, t_4$.

\begin{figure}
\centering
\includegraphics[width=0.7\linewidth]{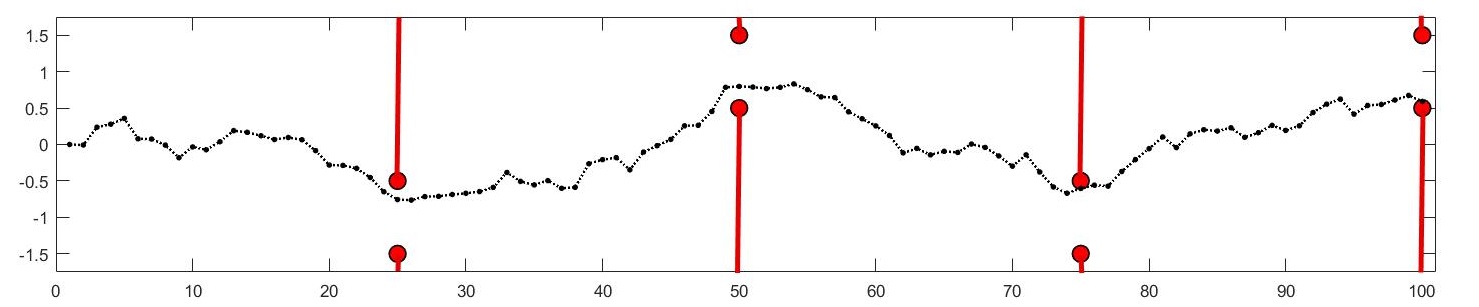}
\caption{Graphical representation of the cost functional \eqref{eq:doors_cost}, which acts as a system of doors through which we want to drive the particle. Superimposed an example of controlled trajectory.} 
\label{f:room_plan}
\end{figure}

Consider a truncated controlled autoregressive process
\[
X_{n+1}=\big(X_{n}+\frac{u_n}{100}+\frac{1}{10}\xi_n\big)\vee-2 \wedge 2 ,\,n=0,\,\dots,\,N,
\]
and the task of guiding the particle through a sequence of doorways, as illustrated in Figure \ref{f:room_plan}, minimising the use of the control. The horizontal axis denotes the time. We postulate the following cost functional which penalises severely for hitting the wall:
\begin{equation}
\label{eq:doors_cost}
J(n,(X_s,u_s)_{s=n}^N)=\sum_{s=n}^N \Big[ bu_s^2+c \sum_{i=1}^4 \mathds{1}_{\{s=t_i\}}\mathds{1}_{\{X_s\notin[d^{-}_{i},d^{+}_{i}]\}} \Big]
\end{equation}
where $b$ represents a quadratic cost for using the control, while $c$ is the penalty for hitting the wall.

For this problem we select the set of basis functions given by $\{1,\,x,\,x^2\}$ and we will test different choices of the training measure $\mu$. We fix $b=1$, $c=100$, $N=100$, and $(t_1, \ldots, t_4)\allowbreak = (25, 50, 75, 100)$.

The purpose of this control problem is to visualise clearly the difference between value and performance iteration. We refer to figure \ref{f:comparison_mu2} to see the effect of the measure $\mu$ on the approximation of a function similar to the shape of each door $\mathds{1}_{\{X_s\notin[d^{-}_{i},d^{+}_{i}]\}}$. We would like to pick the one that induces the best policy, i.e. the one under which the distribution of the controlled process is closest to the one of the optimally controlled process.

\begin{figure}
\centering
\includegraphics[width=1\linewidth]{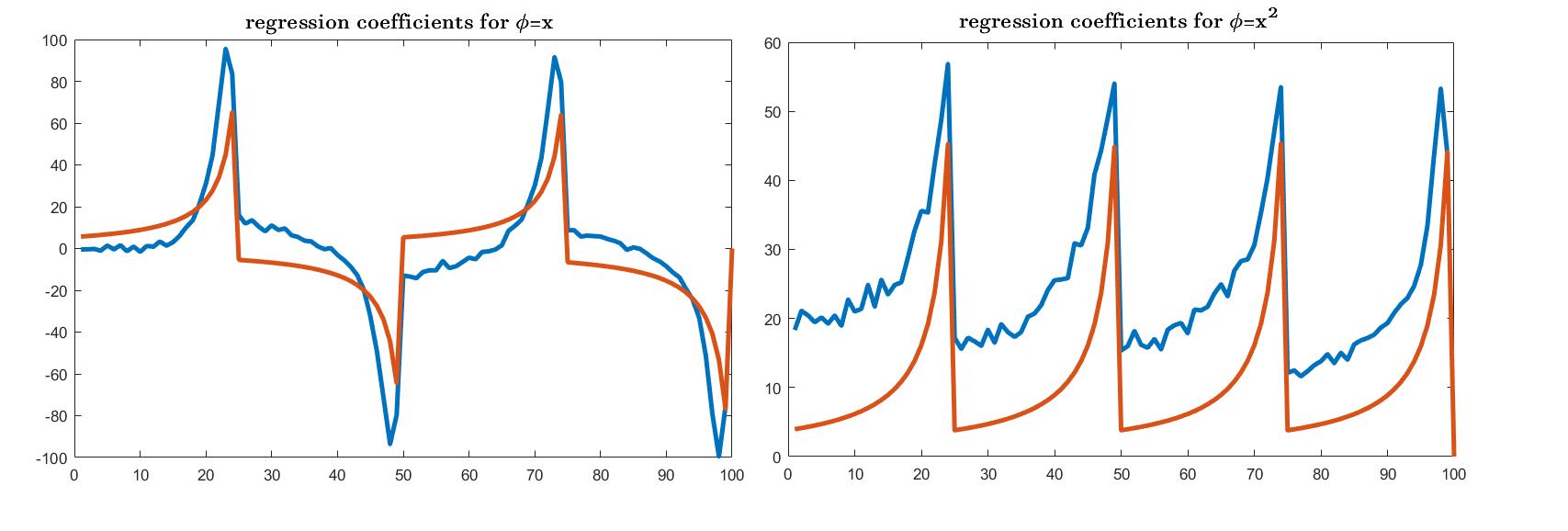}
\includegraphics[width=0.9\linewidth]{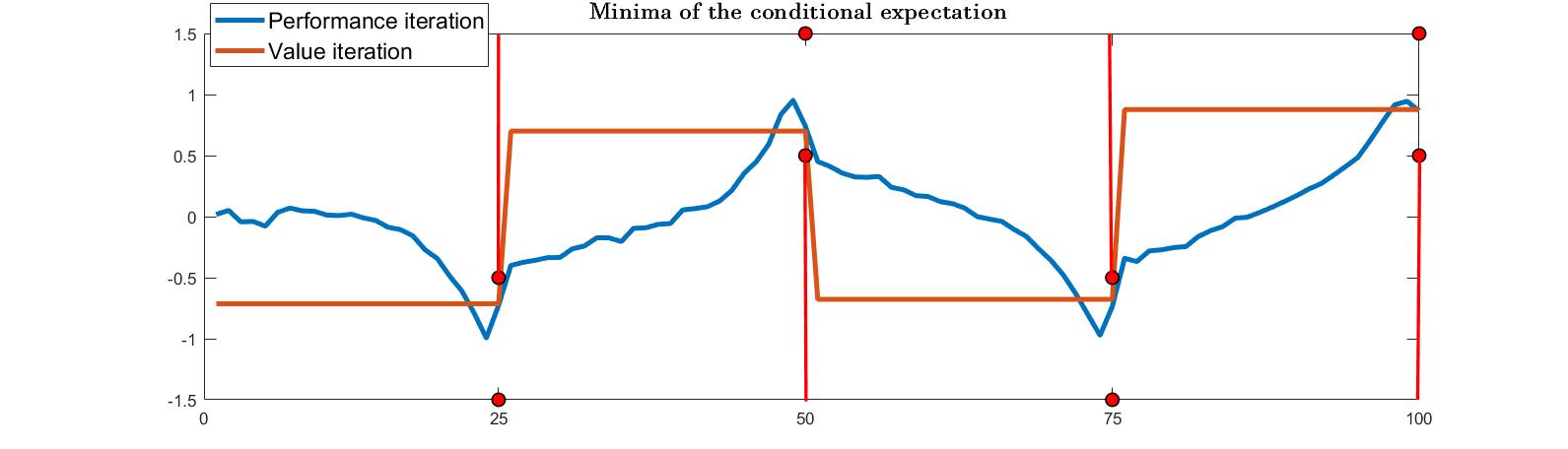}
\caption{The top graphs display regression coefficients for the basis function $x$ and $x^2$ for value and performance iteration. We do not show the coefficient for the constant basis function as it does not play any role in computation of the control. The bottom graph shows the location of the minimum point of the conditional expectation of the estimated continuation value $x \mapsto \hat \alpha_2 \hat \phi_2(x, 0) + \hat \alpha_3 \hat \phi_3(x, 0)$ (recall that the control is additive).} 
\label{f:doors_comparison}
\end{figure}

In order to improve the numerical results, and given the peculiar structure of this problem, we introduce a time dependent training measure $\mu_n$. The intuition is that the measure we choose should help to guide the training points through the rooms, inducing an effective policy. The proofs we presented in Section \ref{proofs} can be adapted to a time dependent training measure $\mu_n$ considering that the only difference is to update the bound $\bar{R}$ with $\bar{R}\max_{n}\big\{\big\|\frac{d\mu_n}{d\mu} \big\|_\infty\big\}$. We use a heuristic technique to generate a training measure $\mu_n$: we first solve the problem using a uniform training measure $\mu$. We then simulate the process using the computed controls and fit a truncated Gaussian distribution at each time $n$. Those distributions are then used as time-dependent training measures $\mu_n$. We iterate this procedure until satisfactory convergence is obtained. For our problem the convergence was quick and required only a few iterations.

\begin{figure}
\centering
\includegraphics[width=1\linewidth]{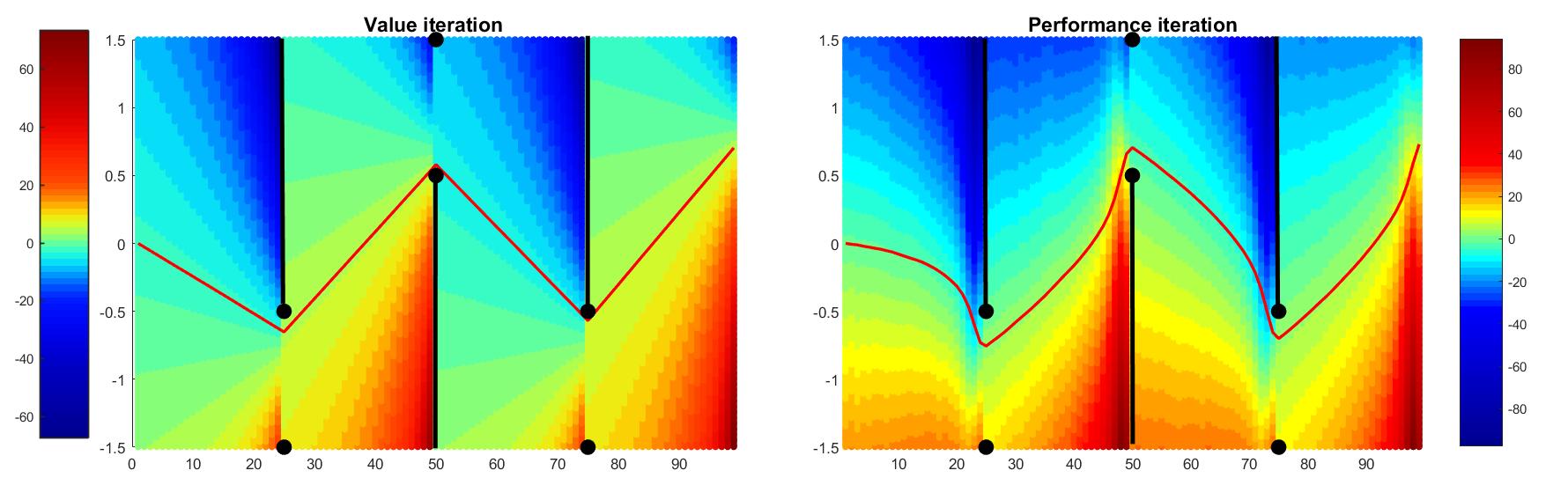}\\
\caption{The two panels display the control policy estimated by the two algorithms; value iteration on the left, performance iteration on the right. Superimposed, in red, the average controlled trajectory. } 
\label{f:doors_comparison_map}
\end{figure}

Our results are displayed in Figures \ref{f:doors_comparison}-\ref{f:performance}. On the top panels in Figure \ref{f:doors_comparison} we notice that the estimates of regression coefficients is smoother for the value iteration scheme than for the performance iteration. This is because of the fact that the conditional expectation of the linear combination of our basis functions together with the cost of applying control (but not the penalty for missing doorways) is exactly representable by the combination of the basis functions. Therefore, in all but 4 times $t_1, \ldots, t_4$, any error is due to Monte Carlo estimation of the coefficients. On the other hand, in the performance iteration scheme not only the variance of coefficient estimates is larger as whole trajectories are generated but also each time the projection is of an estimate of a value function which is not representable by the basis functions due to the presence of doorways and the effect of future controls. However, as displayed in the bottom panel of Figure \ref{f:doors_comparison}, the value iteration scheme is less able to guide the process in the right direction. Indeed, this graph displays the location of the minimum of the mapping $x \mapsto \hat \alpha_2 \hat \phi_2(x, 0) + \hat \alpha_3 \hat \phi_3(x, 0)$, i.e., the point to which control would shift the process if there was no cost involved. The location of the minimum is constant between doorways for the value iteration scheme\footnote{It should be noted that the conditional expectation itself does change over time even though the minimum stays constant, c.f. regression coefficients in the top panels of of Figure \ref{f:doors_comparison}, which explains why the control map in Figure \ref{f:doors_comparison_map} is not piecewise constant.} suggesting insufficient adaptability of the estimated conditional expectations.  In addition, it is close to the boundary of the doorways, resulting in the process often failing to fit through.  On the other hand, the estimates for performance iteration, although more conservative and inducing a higher cost of control, would guide the process more efficiently through doorways. Figure \ref{f:doors_comparison_map} displays control maps for both algorithms which further support the conclusions drawn above. 

This effect of estimated controls is shown on Figure \ref{f:performance}. The left panel displays Monte Carlo estimates of the true performance of estimated policies starting from $X_0=0$. The performance of the value iteration policy is significantly inferior to that obtained through the performance iteration. The right panel explains that difference by showing  the empirical distribution of the pathwise performance for the initial point $X_0 = 0$ of the two policies. The humps correspond to missing 0, 1, 2, and 3 doorways. Trajectories controlled by the performance iteration policy rarely miss more than 1 doorways and with negligible probability miss more than 2 doorways. Conversely, the process controlled by the value iteration policy has a large probability of missing 1 doorway, a significant probability of missing 2 and visible chance of missing 3 doorways. We can shed further light on the properties of the estimated controls. As value iteration humps are located leftwards from the performance iteration ones, we conclude that the performance iteration invests more in controls to avoid frequent impacts with the wall, while the value iteration saves on control but experiences more impacts with the wall, causing a much higher average value.


\begin{figure}
\centering
\includegraphics[width=1\linewidth]{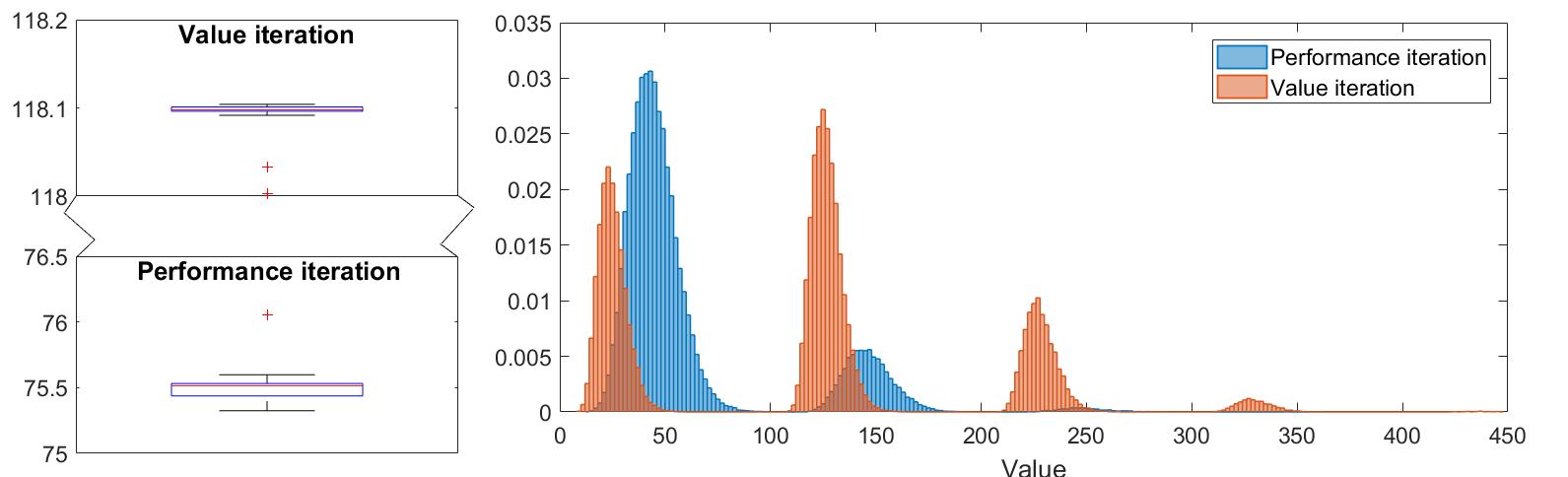}
\caption{
The left panel shows the Monte Carlo estimate of performance at time $0$ of estimated strategies when the controlled process starts from $x_0=0$. 
The right panel displays the empirical histogram of the performance of the policies estimated by the two algorithms. }
\label{f:performance}
\end{figure}

\section{Conclusions}
\label{s:conclusion}

In this paper we have presented a mathematical framework within which the description of Regression Monte Carlo algorithms is very natural. We introduced a general description of the Regress Later algorithm in both value and performance iteration specification for stochastic control problems. Exploiting our mathematical framework we derived the speed of convergence of the two schemes, and in doing so we proved that both schemes converge. We discussed some theoretical and practical consequences of our convergence theorems and finally presented numerical examples that showcase the defferences between the value and performance iteration schemes. The contribution of the paper is at least twofold: we proved the convergence of a Regression Monte Carlo scheme for control of Markov processes, which, to the best of our knowledge, has not been done before; we presented theoretical and practical explanations of the different estimation quality of the value and performance iteration.

Future work should include a study of the optimal choice of time dependent training measures and basis functions, which will allow to successfully apply Regress Later Monte Carlo to an even broader class of problems.

\bibliographystyle{abbrvnat}
\bibliography{reference}{}

\end{document}